\documentclass{amsart}
\usepackage{latexsym}
\usepackage{rotating}
\usepackage{amsfonts}
\usepackage{amssymb}
\usepackage{amsmath}
\usepackage{graphics, color}
\usepackage{amsthm}
\input xy
\xyoption{all}
\DeclareMathOperator*{\colim}{colim}

\DeclareMathOperator*{\tele}{telescope}
\DeclareMathOperator*{\holim}{holim}

\DeclareMathOperator*{\coprodmo}{\coprod}

\newtheorem{theorem}{Theorem}[section]
\newtheorem{lemma}[theorem]{Lemma}
\newtheorem{corollary}[theorem]{Corollary}
\newtheorem{proposition}[theorem]{Proposition}

\newtheorem{example}[theorem]{Example}
\newtheorem{definition}[theorem]{Definition}
\newtheorem{remark}[theorem]{Remark}

\newtheorem*{notation}{Notation and Conventions}

\begin{document}

\newcommand{\bbA}{\mathbb{A}}
\newcommand{\bbC}{\mathbb{C}}
\newcommand{\C}{\mathbb{C}}
\newcommand{\N}{\mathbb{N}}
\newcommand{\bbN}{\mathbb{N}}
\newcommand{\bbR}{\mathbb{R}}
\newcommand{\bbRP}{\mathbb{RP}}
\newcommand{\Z}{\mathbb{Z}}
\newcommand{\bbZ}{\mathbb{Z}}
\newcommand{\mcZ}{\mathcal{Z}}
\newcommand{\bbF}{\mathbb{F}}
\newcommand{\bbQ}{\mathbb{Q}}

\newcommand{\A}{\mathcal{A}}
\newcommand{\B}{\mathcal{B}}
\newcommand{\mcC}{\mathcal{C}}
\newcommand{\mcD}{\mathcal{D}}
\newcommand{\E}{\mathcal{E}}
\newcommand{\mF}{\mathcal{F}}
\newcommand{\G}{\mathcal{G}}
\newcommand{\mcG}{\mathcal{G}}
\newcommand{\mH}{\mathcal{H}}
\newcommand{\mcH}{\mathcal{H}} %Holonomy map
\newcommand{\I}{\mathcal{I}}
\newcommand{\mL}{\mathcal{L}}
\newcommand{\mcN}{\mathcal{N}}
\newcommand{\M}{\mathcal{M}}
\newcommand{\mO}{\mathcal{O}}
\newcommand{\mcP}{\mathcal{P}}
\newcommand{\mcR}{\mathcal{R}}
\newcommand{\T}{\mathcal{T}}
\newcommand{\U}{\mathcal{U}}
\newcommand{\V}{\mathcal{V}}

\newcommand{\bS}{\mathbf{S}}
\newcommand{\bd}{\mathbf{d}}

\newcommand{\bG}{\mathcal{G}_0} %Based gauge group
\newcommand{\sth}[1]{#1^{\mathrm{th}}}
\newcommand{\abs}[1]{\left| #1\right|}
\newcommand{\ord}[1]{\Delta \left( #1 \right)}
\newcommand{\leqs}{\leqslant}
\newcommand{\geqs}{\geqslant}
\newcommand{\heq}{\simeq}
\newcommand{\iso}{\simeq}
\newcommand{\maps}{\longrightarrow}
\newcommand{\lmaps}{\longleftarrow}
\newcommand{\injects}{\hookrightarrow}
\newcommand{\homeo}{\cong}
\newcommand{\surjects}{\twoheadrightarrow}
\newcommand{\isom}{\cong}
\newcommand{\cross}{\times}
\newcommand{\normal}{\vartriangleleft}
\newcommand{\wt}[1]{\widetilde{#1}} %wide tilde for M's
\newcommand{\fc}{\mathcal{A}_{\mathrm{flat}}} %Space of flat connections
\newcommand{\flc}{\mathcal{A}_{\mathrm{fl}}}

\newcommand{\Rdef}{R^{\mathrm{def}}}
\newcommand{\Sym}{\textrm{Sym}}
\newcommand{\vect}[1]{\stackrel{\rightharpoonup}{\mathbf #1}}
\newcommand{\SR}{\mathcal{SR}}
\newcommand{\SRe}{\mathcal{SR}^{\mathrm{even}}}
\newcommand{\Rep}{\mathrm{Rep}}
\newcommand{\SRep}{\mathrm{SRep}}
\newcommand{\Hom}{\mathrm{Hom}}
\newcommand{\HHom}{\mathcal{H}\mathrm{om}}
\newcommand{\Lie}{\mathrm{Lie}}
\newcommand{\K}{K^{\mathrm{def}}}
\newcommand{\mK}{\mathcal{K}_{\mathrm{def}}}
\newcommand{\SK}{SK_{\mathrm{def}}}
\newcommand{\dom}{\mathrm{dom}}
\newcommand{\codom}{\mathrm{codom}}
\newcommand{\Ob}{\mathrm{Ob}}
\newcommand{\Mor}{\mathrm{Mor}}
\newcommand{\+}[1]{\underline{#1}_+}
\newcommand{\Fin}{\Gamma^{\mathrm{op}}}
\newcommand{\f}[1]{\underline{#1}}
\newcommand{\hofib}{\mathrm{hofib}}
\newcommand{\Stab}{\mathrm{Stab}}
\newcommand{\Css}{\mathcal{C}_{ss}}
\newcommand{\Map}{\mathrm{Map}}
\newcommand{\bMap}{\mathrm{Map_*}}
\newcommand{\bdMap}{\mathrm{Map_*^\delta}}
\newcommand{\flatc}{\mathcal{A}_{\mathrm{flat}}}
\newcommand{\F}[1]{\mathrm{Flag}(\vect{#1})}
\newcommand{\p}{\vect{p}}
\newcommand{\avg}{\mathrm{avg}}
\newcommand{\smsh}[1]{\ensuremath{\mathop{\wedge}_{#1}}}
\newcommand{\Vect}{\mathrm{Vect}}
\newcommand{\bv}{\bigvee}
\newcommand{\Gr}{\mathrm{Gr}}
\newcommand{\Mf}{\mathcal{M}_{\textrm{flat}}}
\newcommand{\ku}{\mathbf{ku}}
\newcommand{\Susp}{\Sigma}
\newcommand{\Id}{\textrm{Id}}
\newcommand{\id}{\textrm{Id}}
\newcommand{\xmaps}{\xrightarrow}
\newcommand{\srm}[1]{\stackrel{#1}{\maps}}
\newcommand{\srt}[1]{\stackrel{#1}{\to}}
\newcommand{\sm}{\wedge}
\newcommand{\conv}{\Rightarrow}
\newcommand{\Tor}{\textrm{Tor}}
\newcommand{\goesto}{\mapsto}
\newcommand{\nd}{\noindent}

\def\co{\colon\thinspace}

\title[Flat connections over surfaces]{The stable moduli space of flat connections over a surface}
\author[D\,A Ramras]{Daniel A. Ramras}
%\givenname{Daniel A}
%\surname{Ramras}
\address{New Mexico State University\\
Department of Mathematical Sciences\\
P.O. Box 30001\\
Department 3MB\\
Las Cruces, New Mexico 88003-8001 }
\email{ramras@nmsu.edu}
\urladdr{http://www.math.nmsu.edu/~ramras/}

\thanks{This work was partially supported by an NSF graduate fellowship and by NSF grants DMS-0353640 (RTG),  DMS-0804553, and DMS-0968766.}

%%\keyword{Atiyah-Segal theorem}
%$\keyword{deformation $K$--theory}
%$\keyword{flat connection}
%$\keyword{Yang-Mills theory}

% \subject{primary}{msc2000}{55N15}
% \subject{primary}{msc2000}{58E15}
% \subject{secondary}{msc2000}{58D27}
% \subject{secondary}{msc2000}{19L41}

 \begin{abstract}
We compute the homotopy type of the moduli space of flat, unitary connections over aspherical surfaces, after stabilizing with respect to the rank of the underlying bundle.
Over the orientable surface $M^g$, we show that this space has the homotopy type of the infinite symmetric product of $M^g$, generalizing a well-known fact for the torus.  Over a non-orientable surface, we show that this space is homotopy equivalent to a disjoint union of two tori, whose common dimension
corresponds to the rank of the first (co)homology group of the surface.  Similar calculations are provided for products of surfaces, and show a close analogy with the Quillen--Lichtenbaum conjectures in algebraic $K$--theory.
The proofs utilize Tyler Lawson's work in deformation $K$--theory, and rely heavily on Yang-Mills theory and gauge theory.
\end{abstract}

\maketitle{}

%%%%%%%%%%%%%%%%%%%%%%%%%%%%%%%%%%%%%%%%%%%%%%%%%%%
%%%%%%%%%%%%%%%%%%%%%%%%%%%%%%%%%%%%%%%%%%%%%%%%%%%

\section{Introduction}

Given a compact, connected Lie group $G$ and a manifold $X$, the moduli space of flat $G$--connections on a principal $G$--bundle $P\to X$ is the space $\fc(P)/\G(P)$, where $\fc(P)$ is the space of flat connections and $\G(P) = \textrm{Aut} (P)$ is the gauge group.  Taking the disjoint union of these spaces over representatives for the isomorphism classes of $G$--bundles, one obtains the moduli space $\Mf (X, G)$ of \emph{all} flat $G$--connections over $X$.  Holonomy provides a mapping
$\fc(P)\to \Hom(\pi_1 X, G)$
which, by Uhlenbeck compactness (see, for example, Wehrheim~\cite{Wehrheim}), induces a homeomorphism 
\begin{equation}\label{Mf}\Mf (X, G) \homeo \Hom(\pi_1 X, G)/G.\end{equation}
Hence this moduli space may also be viewed as the moduli space of representations.
We  will study the homotopy type of these moduli spaces when $X = S$ is an aspherical, compact surface and $G = U(n)$.  The inclusions $U(n)\injects U(n+1)$ allow us to stabilize with respect to the rank $n$, and our main result determines the homotopy type of $\Mf(S) = \colim_n \Mf(S, U(n))$.
It follows from (\ref{Mf}) that $\Mf(S)$ is homeomorphic to $\Hom(\pi_1 S, U)/U$, where $U = \colim_n U(n)$ is the infinite unitary group.  

Recall that for a based space $(X, *)$, the infinite symmetric product $\Sym^\infty(X)$ is the colimit $\colim_{n\to \infty} X^n/\Sigma_n$ along the maps 
$[x_1, \ldots, x_n] \mapsto [x_1, \ldots, x_n, *]$.

\vspace{.1in}
\noindent {\bf Theorem \ref{moduli}}\quad
{\sl Let $M^g$ denote an orientable surface of genus $g>0$, with a chosen basepoint $m_0\in M$.  Then there is a homotopy equivalence
$$\Mf(M^g) \heq \Sym^\infty (M^g).$$

If $\Sigma$ is an aspherical, non-orientable surface, there is a homotopy equivalence
$$\Mf(\Sigma) \heq T^{k} \coprod T^k$$
where $k = \textrm{rank} (H^1 (\Sigma))$ and $T^k$ denotes the $k$--dimensional torus.
}

As explained in Section~\ref{coarse-moduli}, this theorem is elementary in the case of the torus $M^1 = S^1\cross S^1$.  For non-orientable surfaces, the homotopy equivalence in the theorem can be described explicitly in several ways (Section~\ref{explicit}).  In particular, the inclusion
$\Hom(\pi_1 \Sigma, U(1)) \injects \Mf(M^g)$ of the group of multiplicative characters is a homotopy equivalence.  We extend these calculations to products of surfaces in Section~\ref{products}, concluding in particular that $\pi_* \Mf(X)\otimes \bbQ \isom H^*(X; \bbQ)$ for such $X$.

These moduli spaces are important in various areas of geometry and topology, due to their close relationship with (semi)-stable holomorphic vector bundles, Yang-Mills theory, Chern--Simons theory, Casson invariants, and gauge theory.  The special case of orientable surfaces has received particular attention.  Atiyah and Bott~\cite{A-B} (and subsequently Daskalopoulos~\cite{Dask} and R{\aa}de~\cite{Rade})
studied these spaces using Morse theory for the Yang--Mills functional.  In particular, Atiyah and Bott provided recurrence relations for the cohomology of the \emph{homotopy} quotient 
$$\left(\fc \right)_{h\G} \heq \Hom\left(\pi_1 M^g, U(n)\right)_{hU(n)}.$$
This work has been refined by Zagier~\cite{Zagier} and extended by Ho and Liu~\cite{Ho-Liu-or-non}. 
The rational cohomology of the moduli space of $SU(2)$ connections over an orientable surface was computed by Cappell, Lee, and Miller~\cite[Theorem 2.2]{CLM}.
In the past few years, there has been substantial interest in the case of non-orientable surfaces.  The connected components of the moduli space $\Hom(\pi_1 \Sigma, G)/G$ (with $\Sigma$ a non-orientable surface) were computed by Ho and Liu~\cite{Ho-Liu-ctd-comp-I, Ho-Liu-ctd-comp-II}, and Tom Baird~\cite{Baird} has provided computations of the Poincar{\'e} series of $\Hom\left(\pi_1 \Sigma, SU(2)\right)/SU(2)$  as well as those of $\Hom(\pi_1 \Sigma, SU(2))$ and $\Hom\left(\pi_1 \Sigma, SU(2)\right)_{hSU(2)}$.  Ho and Liu have extended Atiyah and Bott's study of Yang--Mills theory to non-orientable surfaces~\cite{Ho-Liu-non-orient, Ho-Liu-anti-perfect}, with the goal of computing these Poincar{\'e} series more generally, and progress along these lines has been made by Baird~\cite{Baird-anti, Baird-Klein}.

Our computation of the stable moduli space uses 
a new technique developed by Tyler Lawson~\cite{Lawson-simul}, which relies on a great deal of homotopy theory.  Lawson has established a close relationship between the stable moduli space $\Mf (X)$ and Carlsson's deformation $K$--theory spectrum~\cite{Carlsson-derived-rep, Lawson-prod, Ramras-excision} of $\pi_1 X$, denoted $\K(\pi_1 X)$.  In fact, these results apply to all finitely generated discrete groups $\Gamma$. 
Lawson has shown that the spectrum $\K(\Gamma)$ carries a Bott map 
$$\Susp^2 \K(\Gamma) \maps \K(\Gamma),$$
inherited from the ordinary Bott map on the connective $K$--theory spectrum $\ku \heq \K(\{1\})$.
Lawson's theorem states that the homotopy cofiber of this map is the group completion of the monoid
$$\coprodmo_{n=0}^\infty \Hom\left(\Gamma, U(n)\right)/U(n).$$
For surface groups, the group completion process corresponds (essentially) to the colimit in the definition of the stable moduli space (see Lemmas~\ref{EM} and~\ref{EM2}).

Calculations of $\pi_* \Mf (S)$ in low degrees were obtained by the author 
in~\cite{Ramras-surface} using Lawson's results in conjunction with Morse theory for the Yang--Mills functional.
  In this paper, we complete the computation using an excision theorem for connected sum decompositions (Theorem~\ref{excision-thm}), which shows that $\K (S)$ may be built up homotopically from the deformation $K$--theory of free groups.  Lawson's earlier computations for free groups~\cite{Lawson-simul} then come into play.  

The excision theorem, and the subsequent results in this paper, rest on an analysis of the natural map $B\co \Hom(\pi_1 S, U(n)) \to \bMap(B\pi_1 S, BU(n))$
sending a representation $\rho$ to the induced map $B\rho$ on classifying spaces.  In Theorem~\ref{B-highly-ctd}, we show that this map is highly connected whenever $S$ is an aspherical surface.
The study of maps between classifying spaces, and their relation to homomorphisms, plays a central role in homotopy theory.  For example, maps between Eilenberg-MacLane spaces classify cohomology operations.  A great number of developments in this area were inspired by H. Miller's proof of the Sullivan conjecture~\cite{Miller}.  For example, Dwyer and Zabrodsky~\cite{D-Z} showed that the natural map
$B\co \Hom( P, G) \maps \bMap(BP, BG)$
induces an isomorphism on $\pi_0$ whenever $P$ is a finite $P$--group and $G$ is a compact Lie group.  Previous work in this area has stayed within pure homotopy theory, but our proof of Theorem~\ref{B-highly-ctd} depends on gauge theory and on Morse theory for the Yang--Mills functional.

Our excision theorem, together with Lawson's product formula~\cite{Lawson-prod}, allows us to give explicit descriptions of the $\ku$--module structure on $\K(\pi_1 X)$ for $X$ a product of aspherical surfaces (Section~\ref{ku}).  The author established an isomorphism in \emph{positive dimensional} homotopy between the deformation $K$--theory of a surface group and the complex $K$--theory of the underlying surface~\cite{Ramras-surface}.
We show that there is in fact a morphism of $\ku$--modules $\K(\pi_1 X) \maps F (X_+, \ku)$
that induces an isomorphism in homotopy in dimensions greater than $\textrm{qcd} (X) - 2$.
Here $F$ denotes the based function spectrum, and $\textrm{qcd} (X)$ is the \emph{rational} cohomological dimension of $X$ (i.e. the largest integer $n$ such that $H^*(X; \bbQ)\neq 0$). 
These result are analogous to the theorem of Atiyah and Segal~\cite{Atiyah-Segal} relating representation of compact Lie groups to $K$--theory.
 (It is important to note that there is not always a relationship between deformation $K$--theory and complex $K$--theory: see for instance Lawson's computations for the integral Heisenberg group given in~\cite{Lawson-simul}.)

The above cohomological bound also appears in the Quillen--Lichtenbaum conjecture in algebraic $K$--theory. This conjecture states that there is an isomorphism $K_n (X; \Z/l)\to K^{\text{\'{e}t}}_n (X; \Z/l)$ from the algebraic $K$--theory of a scheme $X$ to its \'{e}tale $K$--theory when $n$ is \emph{larger} than
cd($X$) - 2.  Here cd($X$) denotes the (\'{e}tale) cohomological dimension.  For a precise formulation (and proof) of the Quillen--Lichtenbaum conjecture at the prime $2$, see \O stv\ae r and Rosenschon~\cite{OR}.  Levine's preprint~\cite{Levine} is another good source of information on this topic. 

Lawson~\cite{Lawson-prod} has given an explicit model for deformation $K$--theory, and its ring structure, in terms of $\Gamma$--spaces (in the sense of Segal).
An alternate approach to the results in this paper would be to construct a (multiplicative) map of $\Gamma$--spaces from Lawson's explicit model to some analogous model for topological $K$--theory.  One would hope to show that such a map induces a weak equivalence of the associated $\ku$--algebras whenever the connectivity of the map $B$ goes to infinity with $n$.  Since our present interest is in the stable moduli space, we prefer the more elementary approach based on excision. 

We will now briefly explain the relationship between this paper and the author's previous article~\cite{Ramras-surface}.  The main result of~\cite{Ramras-surface} was an isomorphism (equivariant under the action of the mapping class group) $\K_*(\pi_1 S) \isom K^{-*} (S)$ (for $*>0$ in the orientable case, and for $*\geqs 0$ in the non-orientable case).  The maps underlying these isomorphisms involved Sobolev spaces of connections, and were not sufficiently natural to identify the Bott map or to establish excision (see the introduction to Section~\ref{excision}).  However, the Yang--Mills theory from~\cite{Ramras-surface} still plays a key role in this paper, appearing here as Proposition~\ref{fc-conn}.  (The analytical results regarding holonomy also play a role in Lemma~\ref{fibs}.)  The results in~\cite{Ramras-surface} lead to computations of $\pi_* \Mf(S)$ in low degrees, and we refer back to these calculations in proving our main result.  We note that the abstract isomorphism $\K_*(S) \isom K^{-*} (S)$ can be deduced from excision (see Remark~\ref{Kdef*}), but the explicit maps used in~\cite{Ramras-surface} may prove to be useful for other reasons.

This paper is structured as follows.  In Section~\ref{ex-sec}, we reduce the excision problem in deformation $K$--theory to a question about representation spaces.  In Section~\ref{universal}, we study the above map $B$.  These results are combined in Section~\ref{excision} to prove an excision theorem for connected sum decompositions of surfaces.  In Section~\ref{coarse-moduli}, we review Lawson's results on the Bott map in deformation $K$--theory and prove our main theorem.  Computations for products of surfaces appear in Section~\ref{ku}, where we also discuss some of the basic features of the theory of $\bS$--modules needed for these computations.

\begin{notation}
Throughout this paper, all surfaces $S$ will be compact.  Apart from our discussion of connected sum decompositions, the theorems in this paper only deal with aspherical surfaces, i.e. those with contractible universal cover. Said another way, we require that $B\pi_1 S \heq S$ (of course, the only compact surfaces that are not aspherical are the sphere and the real projective plane).   We use $M^g$ to denote an orientable surface of genus $g$, and if $\Sigma$ is a non-orientable surface, then the genus of its orientable double cover will be denoted by $\wt{g}$.  

We implictly work in the category of compactly generated spaces, meaning that we replace the topology on any space in question with the compactly generated topology.  In particular, products and mapping spaces (between compactly generated spaces) are given their associated compactly generated topologies.  The Sobolev spaces of connections and gauge transformations appearing in this paper are (presumably) not compactly generated, and we will phrase our continuity arguments in the ordinary topology, implicitly replacing the topologies after the fact.
\end{notation}

\noindent
 {\bf Acknowledgments.}
Some of the results in Section 2 first appeared in my Stanford University Ph.D. thesis~\cite{Ramras-thesis}, directed by Gunnar Carlsson.  I would like to thank Tyler
Lawson for helpful discussions regarding Section~\ref{ku}, and for his comments on an earlier draft.  I would also like to thank the referee, whose comments and suggestions greatly improved the exposition, and Sean Lawton for a helpful discussion regarding Lemma 5.7.
This article was written at Columbia University, and I would like to thank the Columbia mathematics department for its hospitality.

%%%%%%%%%%%%%%%%%%%%%%%%%%%%%%%%%%%%%%%%%%%%%%%%%%%
%%%%%%%%%%%%%%%%%%%%%%%%%%%%%%%%%%%%%%%%%%%%%%%%%%%

\section{Excision in deformation $K$--theory}$\label{ex-sec}$

In this section we will study the behavior of deformation K--theory on amalgamated products.  
We begin by briefly reviewing deformation $K$--theory.  For further details and discussion, see~\cite{Lawson-prod, Ramras-excision, Ramras-surface}.

\subsection{The deformation $K$-theory spectrum and its zeroth space}$\label{zeroth-sp}$

The deformation $K$--theory spectrum $\K (\Gamma)$ is the group completion of the topological monoid
of homotopy orbit spaces
$$\Rep(\Gamma)_{hU} := \coprodmo_{n=0}^{\infty}  \Hom\left(\Gamma, U(n)\right)_{hU(n)} 
= \coprodmo_{n=0}^{\infty} EU(n)\cross_{U(n)} \Hom(\Gamma, U(n)),$$
and may be thought of as the homotopical analogue of the representation ring $R(\Gamma)$ (the ordinary group completion of the discrete monoid of isomorphism classes of representations).
For example, when $\Gamma = \{1\}$, deformation $K$--theory is simply the group completion of the monoid
$\coprod BU(n)$, and hence $\K(\{1\}) = \mathbf{ku}$, the connective $K$--theory spectrum.

It is convenient to introduce the topological monoid
$$\Rep(\Gamma) = \coprodmo_{n=0}^{\infty} \Hom(\Gamma, U(n)).$$
We say $\Rep(\Gamma)$ is \emph{stably group-like} with respect to a representation $\psi_0$ if for every representation $\rho$ there is representation $\rho^{*}$ such that $\rho\oplus \rho^{*}$ lies in the same component as $n \psi_0 = \psi_0 \oplus \cdots \oplus \psi_0$.  For the surface groups studied in this paper, $\Rep(-)$ is stably group-like with respect to the trivial representation $1\in \Hom(-, U(1))$ (see Ramras~\cite[Section 4]{Ramras-surface} or Ho and Liu~\cite{Ho-Liu-ctd-comp-II}), and we simply say that $\Rep(-)$ is stably group-like.

The following result, which is an application of the McDuff--Segal group completion theorem~\cite{McDuff-Segal}, gives a concrete model for the zeroth space of the deformation $K$--theory spectrum under the assumption that $\Rep(\Gamma)$ is stably group-like with respect to some representation $\psi_0$.

\begin{proposition}[Ramras~\cite{Ramras-excision}]$\label{gp-completion}$
Let $\Gamma$ be a finitely generated discrete group such that $\Rep(\Gamma)$ is stably group-like with respect to a representation $\psi_0 \in \Hom(\Gamma,U(k))$.
Then there is a zig-zag of weak equivalences between the zeroth space of $\K(\Gamma)$ and the mapping telescope
$$\tele \left( \Rep(\Gamma)_{hU} \stackrel{\oplus \psi_0}{\maps} \Rep(\Gamma)_{hU} 
			\stackrel{\oplus\psi_0}{\maps} \cdots \right),
$$
where $\oplus \psi_0$ denotes block sum with the point $[*_k, \psi_0]\in \Hom\left(\Gamma, U(k)\right)_{hU(k)}$.  \end{proposition}

\subsection{Reduction of excision to representation spaces}$\label{reduction}$

In this section we consider the behavior of deformation $K$--theory on amalgamated products of groups, making crucial use of Proposition~\ref{gp-completion}.    Given an amalgamation diagram of groups, applying deformation $K$-theory results in a pull-back diagram of spectra.  An excision theorem states that the natural map
$$\Phi\co \K_* ( G*_K H) \maps \pi_* \mathrm{holim} \Big( \K (G) \maps \K (K) \longleftarrow  \K (H) 
\Big)$$
is an isomorphism (at least in a range of dimensions), where holim denotes the homotopy pullback.  Throughout this section we will, by abuse of notation, consider the symbol $G*_K H$ to denote an amalgamation \emph{diagram} of groups, as well as the pushout of this diagram.

Our goal is to reduce the excision problem to a question about representation spaces.  We wish to show that information about the maps
$$\xymatrix{ \Hom(G*_K H, U(n)) \ar[d] \\
	 \holim \Big(\Hom(G, U(n)) \maps
								      		 \Hom(K, U(n)) \lmaps
					      					 \Hom(H, U(n))
						    \Big)
		}
$$
allows us to deduce information about the map
$$ \K (G*_K H)
	\stackrel{\Phi}{\maps} \holim \Big(\K(G) \maps
							\K(K) \lmaps
					      		\K(H)
						    \Big).
$$

In order to pass from representation spaces to deformation $K$-theory, we first need to deduce results about homotopy orbit spaces.  This will be done by studying the fibration $X\to X_{hG} \to BG$ associated to a $G$-space $X$.  We need a simple fact about fibrations and homotopy limits.  In order to state the result in an appropriate form, we make the following definition.

\begin{definition}$\label{k,l-ctd}$
We call a map $f\co (X, x_0) \to (Y, y_0)$ of based spaces (or spectra) $(l, k)$--connected ($0\leqs l \leqs k$) if
$f_* \co \pi_n (X, x_0)\to \pi_n (Y, y_0)$ is an isomorphism for $l\leqs n < k$, a surjection for $n = k$, and an injection for
$n = l-1$.  We call a commutative square of based spaces (or spectra)
$$\xymatrix{
	X \ar[r]^\beta \ar[d]^\alpha & Z \ar[d]^\eta \\
	Y \ar[r]^\gamma & W
		}
$$	
$(l, k)$--cartesian if the natural map $X\to \holim \left(Y\maps Z \longleftarrow W \right)$
is $(l, k)$--connected.  If this natural map is a weak equivalence, we call the square \emph{homotopy cartesian}.  Here the basepoint in the homotopy limit is the triple $(y_0, c_{w_0}, z_0)$, where $c_{w_0}$ denotes the constant path.
\end{definition}

We allow $k = \infty$ and $l = 0$, and we set $\pi_{-1} Z = 0$ for any based space $Z$, so that $(0,k)$--connectivity is the standard notion of $k$--connectivity.
The above definition is useful since in certain cases (e.g. connected sum decompositions of Riemann surfaces) excision will hold only above some dimension.  We will, in the end, be considering infinite loop spaces, where all components are homotopy equivalent.  Thus it will suffice to work with a single basepoint, even though many of the spaces below are disconnected.

We now recall that associated to any homotopy pullback (of based spaces or spectra)
$$H = \holim (Y \stackrel{\gamma}{\maps} W \stackrel{\eta}{\lmaps} Z)$$
there is a natural long exact Mayer--Vietoris sequence in homotopy:
\begin{equation}\label{M-V}
\cdots \stackrel{\partial}{\maps} \pi_* H \maps \pi_* Y \oplus \pi_* Z \xmaps{\gamma_* - \eta_*}
\pi_* W \stackrel{\partial}{\maps} \pi_{*-1} H \maps \cdots.
\end{equation}
The map $\pi_* H \to \pi_*Y \oplus \pi_* Z$ is simply the direct sum of the maps induced by the natural projections
$H \to Y$ and $H \to Z$, and the boundary map is induced by the natural map
$\Omega W \to H$.
Any commutative diagram
\begin{equation}\label{MV-diag}
\xymatrix{ Y \ar[r] \ar[d] & W \ar[d] & Z \ar[l] \ar[d] \\
		   Y' \ar[r] & W' & Z' \ar[l]
		   }
\end{equation}
induces a map
$$\holim (Y\to W\leftarrow Z) \maps \holim (Y' \to W' \leftarrow Z'),$$
and there is an associated commutative diagram of Mayer--Vietoris sequences.  It follows that if the vertical maps in (\ref{MV-diag}) are weak equivalences, so is the induced map between the homotopy pullbacks.  We will use this fact several times.

Given a homotopy cartesian square as in Definition~\ref{k,l-ctd}, we may replace the terms $\pi_*H$ in the Mayer--Vietoris sequence (\ref{M-V}) by $\pi_*X$, obtaining a natural, long exact Mayer--Vietoris sequence
\begin{equation}\label{M-V2}\cdots \stackrel{\partial}{\maps} \pi_* X \xmaps{\alpha_*\oplus \beta_*} \pi_* Y \oplus \pi_* Z \xmaps{\gamma_* - \eta_*}
\pi_* W \stackrel{\partial}{\maps} \pi_{*-1} X  \xmaps{\alpha_*\oplus \beta_*} \cdots.
\end{equation}
Similarly, any $(l,k)$--cartesian square has an associated (partial) Mayer--Vietoris sequence, which exists in dimensions $l\leqs * \leqs k$.\footnote{The converses of these statements also hold, in a sense.  If a square of based spaces admits a Mayer--Vietoris sequence of the form (\ref{M-V2}) with sufficiently natural boundary maps, then the square is homotopy cartesian.  Here ``sufficiently natural" simply means that the square relating these boundary maps to the boundary maps in the sequence (\ref{M-V}) must commute; the 5-lemma then shows that the diagram is homotopy cartesian.  A similar statement can be made for partial Mayer--Vietoris sequences and $(l,k)$--cartesian diagrams.}

\begin{lemma}$\label{htpy-limits}$
Let $\G$ be a connected group and let $X$, $Y$, and $Z$ be based $\G$--spaces (we do not assume $\G$ fixes the basepoints).  Then a commutative square of based, equivariant maps  
$$\xymatrix{
	X \ar[r] \ar[d] & Z \ar[d] \\
	Y \ar[r] & W
		}
$$	
is $(l, k)$--cartesian if and only if the diagram of homotopy orbit spaces
$$\xymatrix{
	X_{h\G} \ar[r] \ar[d] & Z_{h\G} \ar[d] \\
	Y_{h\G} \ar[r] & W_{h\G}
		}
$$
is $(l, k)$--cartesian.  (Here we use $[*, x_0]$ as the basepoint of $X_{h\G}$, where $* \in E\G$ is a chosen basepoint and $x_0$ is the basepoint in $X$, and similarly for the other homotopy orbit spaces.)
\end{lemma}
{\bf Proof.}
Consider the commutative diagram of fibrations
$$\xymatrix@!0{
	{} & Z \ar[rr] \ar'[d][dd] & {} &  Z_{h\G} \ar[rr] \ar'[d][dd] & {} & B\G \ar[dd] \\
	X \ar[ur] \ar[rr] \ar[dd] & {} &X_{h\G} \ar[ur] \ar[rr] \ar[dd] & {}
		& B\G  \ar[ur] \ar[dd] \\
	{} & W \ar'[r][rr] & {} & W_{h\G}\ar'[r][rr] & {} &B\G, \\
	Y \ar[ur] \ar[rr] & {} & Y_{h\G} \ar[ur] \ar[rr] & {}
		& B\G  \ar[ur]
			}
$$
in which all the maps $B\G\to B\G$ are the identity.
Let $\widetilde{B\G}$ denote the homotopy limit
$\holim \left( B\G \stackrel{=}{\to} B\G \stackrel{=}{\leftarrow} B\G \right)$, and note that the natural map $B\G\stackrel{\heq}{\to} \widetilde{B\G}$ is a weak equivalence.
Let
$$X \stackrel{\phi}{\maps} \holim \left( Y \to W \leftarrow Z \right) \mathrm{\,\, and \,\,}
X_{h\G} \stackrel{\Phi}{\maps} \holim \left( Y_{h\G} \to W_{h\G} \leftarrow Z_{h\G} \right)$$
denote the natural maps.
Consider the diagram
$$\xymatrix{
	X \ar[r] \ar[d]^{\phi} & X_{h\G} \ar[r] \ar[d]^{\Phi} & B\G \ar[d]^{\heq} \\
	{ \holim  \left( Y\to W \leftarrow Z \right)} \ar[r] \ar@{..>}[d]^{\iota}
	 	&  {\holim \left( Y_{h\G} \to W_{h\G} \leftarrow Z_{h\G} \right)} \ar[r]^(.75){\alpha} \ar[d]^{\heq}
		& {\widetilde{B\G}}, \\
	{\hofib (\alpha)} \ar[r] & P_{\alpha} \ar[ur]
		}
$$
in which $P_{\alpha}$ denotes the total space of the fibration associated to $\alpha$, and $\iota$ exists because the composite along the middle row is constant.  We claim that the map $\iota$ is a weak equivalence, i.e. that the middle row is a homotopy fibration.  Assuming this, the lemma follows easily by applying the (strong) 5--Lemma to the resulting diagram of long exact sequences.  (Note that since we have assumed $\G$ is connected, $\pi_1 B\G = 0$.  Hence these long exact sequences can be cut off after the $\pi_1$ stage, and we need not worry about applying the five lemma to a diagram containing sets.  Moreover, $\pi_0$ is easily dealt with since $\pi_1 B\G = \pi_0 B\G = 0$ implies that the maps on $\pi_0$ induced by $\phi$ and $\Phi$ are isomorphic.)

To see that $\iota$ is a weak equivalence, note for any $\G$--space $T$, the natural inclusion
$T \injects \hofib( T_{h\G} \to B\G)$ is a weak equivalence, and hence the induced map
$$\holim \left( \vcenter{\xymatrix{Y\ar[d]\\ W \\ Z\ar[u]} } \right)
\stackrel{\Psi}{\maps}
\holim \left( \vcenter{\xymatrix{
				\hofib(Y_{h\G}\to B\G) \ar[d] \\
				\hofib(W_{h\G}\to B\G) \\
				\hofib (Z_{h\G}\to B\G) \ar[u]
					     }
				}
	\right)
$$
is a weak equivalence as well.
Now $\iota$ is simply the composition of $\Psi$ with the natural homeomorphism
$$\holim \left( \vcenter{\xymatrix{
					\hofib(Y_{h\G}\to B\G) \ar[d] \\
					\hofib(W_{h\G}\to B\G) \\
					\hofib (Z_{h\G}\to B\G) \ar[u]
						}
				}					
		\right)
	\stackrel{\homeo}{\maps}
		\hofib \left( \holim  \left( \vcenter{\xymatrix{
								 Y_{h\G}\ar[d]\\ W_{h\G} \\ Z_{h\G} \ar[u]
								 		}
								}
						\right)
				\srm{\alpha}
				\holim \left( \vcenter{\xymatrix{B\G \ar[d]^{\begin{turn}{90} $=$ \end{turn}}\\
									       B\G \\
									       B\G \ar[u]_{\begin{turn}{90} $=$ \end{turn}}
									      }
								}
				      \right)
			\right),
$$
so $\iota$ is a weak equivalence as well.
$\hfill \Box$

\vspace{.15in}
We are now ready to discuss our reduction of the excision problem to representation spaces.  Given an amalgamation diagram
$$\xymatrix{
       {K} \ar[r]^{i_1} \ar[d]^{i_2}
       & {H} \ar[d]^{f_1} \\
       {G} \ar[r]^(.36){f_2}
       & {G*_K H,}
                     }
$$
we say that $\Rep(G*_K H)$ is \emph{compatibly stably group-like} if there is a representation 
$\psi \co G*_K H \to U(n)$ such that $\Rep(G*_H K)$ is stably group-like with respect to $\psi$ and the monoids $\Rep(G)$, $\Rep(H)$, and $\Rep(K)$ are stably group-like with respect to the various restrictions of $\psi$.  

\begin{proposition}$\label{reduction-prop}$
Assume that $\Rep(G*_K H)$ is compatibly stably-grouplike.
If the natural map
$$\xymatrix{ \Hom(G*_K H, U(n)) \ar[d]^\phi \\
	 \holim \left( \Hom(G, U(n)) \stackrel{i_1^*}{\maps}
				\Hom(K, U(n)) \stackrel{i_2^*}{\lmaps}
				\Hom(H, U(n))
		\right)
		}
$$
is $(l, k)$--connected for infinitely many $n$, then the natural map
$$\K (G*_K H) \stackrel{\Phi}{\maps}
	 \holim \left( \K (G)\stackrel{i_1^*}{\maps}  \K (K) \stackrel{i_2^*}{\lmaps}
					      		 \K (H)
		\right)
$$
is $(l, k)$--connected as well.  (We use the trivial representations as our basepoints.)
\end{proposition}
\begin{proof} Since $\Rep(G*_K H)$ is compatibly stably group-like with respect to a representation 
$\psi: G*_K H \to U(n)$, Proposition~\ref{gp-completion} allows us to replace the diagram of deformation $K$--theory spectra with the following diagram of infinite mapping telescopes:
\begin{equation}\label{telescopes}
\xymatrix{
	\tele \limits_{\stackrel{\oplus \psi}{\maps}} \left( \Rep(G*_K H)_{hU} \right) \ar[r] \ar[d]
	& \tele \limits_{\stackrel{\oplus \psi|_H}{\maps}} \left( \Rep(H)_{hU} \right) \ar[d] \\
	\tele \limits_{ \stackrel{\oplus \psi|_G}{\maps}} \left( \Rep(G)_{hU} \right) \ar[r]
	& \tele \limits_{\stackrel{\oplus \psi|_K}{\maps} } \left( \Rep(K)_{hU} \right).
		}
\end{equation}
On positive-dimensional homotopy groups, the desired conclusions about the map $\Phi$ now follow from Lemma~\ref{htpy-limits}.  To handle $\pi_0$, we must be slightly more careful.

Each of the telescopes in (\ref{telescopes}) is fact a disjoint union of simpler telescopes.  Given a natural number $k$, a group $\Gamma$, and a unitary representation $\rho$ of $\Gamma$, we define
$\tele_k (\Gamma, \rho)$ to be the mapping telescope
$$\tele \left( \Hom\left(\Gamma, U(k)\right)_{hU(k)}
\stackrel{\oplus \rho}{\maps} \Hom\left(\Gamma, U(k+\dim(\rho))\right)_{hU(k+\dim (\rho))}
\stackrel{\oplus \rho}{\maps} \cdots \right).
$$
By collapsing stages of these telescopes, we see that the canonical inclusions
$$\textrm{telescope}_{k+m\dim (\rho)} (\Gamma, \rho) \injects \textrm{telescope}_k (\Gamma, \rho)$$
are homotopy equivalences for $m = 1, 2, \ldots$.
Hence the inclusion
$$\tele_{\stackrel{\oplus \rho}{\maps}} \left( \Rep(\Gamma)_{hU} \right)
\injects \coprodmo_{k = 0}^{\dim \rho - 1} \bbZ \cross \textrm{telescope}_k (\Gamma, \rho)
$$
is a homotopy equivalence.
Applying these observations to the diagram (\ref{telescopes}), we obtain a \emph{commutative} cube linking that square of telescopes to the square
\begin{equation}\label{telescopes2}
\xymatrix{
	\coprod \limits_{k=0}^{\dim (\rho) - 1} \bbZ \cross \tele_k (G*_K H, \psi) \ar[r] \ar[dd]
	& \coprod \limits_{k=0}^{\dim (\rho) - 1} \bbZ \cross \tele_k (H, \psi|_H)  \ar[dd] \\
	& \\
	\coprod \limits_{k=0}^{\dim (\rho) - 1} \bbZ \cross \tele_k (G, \psi|_G)  \ar[r]
	& \coprod \limits_{k=0}^{\dim (\rho) - 1} \bbZ \cross \tele_k (K, \psi|_K).
		}
\end{equation}
Since the maps in (\ref{telescopes2}) preserve both the disjoint union over $k$ and the disjoint union over $\bbZ$, the homotopy pullback of this diagram is precisely
$$\coprodmo_{k=0}^{\dim (\rho) - 1} \bbZ \cross
\holim \left( \vcenter{\xymatrix{
       							\tele_k (H, \psi|_H) \ar[d]\\
					      		\tele_k (K, \psi|_K) \\
					      		\tele_k (G, \psi|_G)    \ar[u]
								}} \right)
$$
and the result follows from Lemma~\ref{htpy-limits}.								
\end{proof}

%%%%%%%%%%%%%%%%%%%%%%%%%%%%%%%%%%%%%%%%%%%%%%%%%%%
%%%%%%%%%%%%%%%%%%%%%%%%%%%%%%%%%%%%%%%%%%%%%%%%%%%
\section{Maps between classifying spaces}$\label{universal}$

In this section we use Morse theory for the Yang--Mills functional and gauge-theoretical constructions to study the connectivity of the natural map
\begin{equation}\label{B}B\co \Hom(\pi_1 S, BU(n)) \maps \Map_* (B\pi_1 S, BU(n))\end{equation}
for $S$ an aspherical surface.  
We show that the connectivity of this map tends to infinity with $n$ (Theorem~\ref{B-highly-ctd}), so long as one considers only null-homotopic maps in the orientable case.
Here $B\pi_1 S$ and $BU(n)$ refer to the simplicial classifying spaces of these groups, as in Segal~\cite{Segal-class-ss}.  

We begin by reviewing the construction of the simplicial classifying space and the mapping (\ref{B}).
For any topological group $G$, the space $BG$ is the geometric realization of the internal category in topological spaces with one object and with $G$ as morphisms, so that $BG$ is the geometric realization of a
 simplicial space $k\mapsto G^k$.
Similarly, $EG$ is defined as the geometric realization of the internal category in topological spaces with object space $G$ and morphism space $G\cross G$; here $(g,h)$ is the unique morphism from $h$ to $g$ and $(k,g)\circ (g,h) = (k, h)$.  Thus $EG$ is the realization of a simplicial space of the form $k\mapsto G^{k+1}$. 
The functor $(g,h)\mapsto g^{-1} h$ gives a map $\pi\co EG\to BG$, which is a model for the universal principal $G$--bundle whenever the identity in $G$ is a non-degenerate basepoint (see May~\cite[Theorem 8.2]{May-CSF}) and in particular if $G$ is discrete or a compact Lie group.  The unique object in the category underlying $BG$ and the object $e\in G$ in the category underlying $EG$ give basepoints $*\in BG$ and $*\in EG$, with $\pi(*) = *$.  For future reference, we note that $EG$ deformation retracts to $*$, because on the underlying category the projection functor $g \mapsto e$ (sending all morphisms to the identity) is naturally equivalent to the identity functor.
 
 Given two topological groups $\Gamma$ and $G$, the corresponding map
 $$B = B_\Gamma \co \Hom(\Gamma, G) \maps \bMap (B\Gamma, BG),$$
is defined by considering $\rho\co \Gamma \to G$ as a functor and taking the induced map on geometric realizations, which is always a based map.  
To see that the map $B$ is continuous, note that it is the adjoint of the map on quotient spaces induced by the map
\begin{equation}\label{simplices}
\Hom(\Gamma, U(n)) \cross \left( \coprod_k \Gamma^k \cross \Delta^k\right) \maps \left( \coprod_k U(n)^k \cross \Delta^k\right)
\end{equation}
that sends $(\rho, \gamma_1, \ldots, \gamma_k, t)$ to $(\rho(\gamma_1), \ldots \rho(\gamma_k), t)$.  Since $\Hom(\Gamma, U(n))$ is topologized as a subspace of the mapping space $\Map(\Gamma, U(n))$ with the compact-open topology, the map (\ref{simplices}) is continuous, and hence $B$ is as well.  Note that there is an analogous continuous mapping
\begin{equation}\label{E}
E\co \Hom(\Gamma, G) \maps \Map_* (E\Gamma, EG).
\end{equation}

\vspace{.15in} 
Our analysis of the map $B$ will rely on the study of \emph{flat connections}.  We briefly review the notion of a flat connection and its holonomy.  For further details on these concepts and other gauge-theoretical material below, we refer the reader to Ramras~\cite{Ramras-surface} or Wehrheim~\cite{Wehrheim}.  

In this section $X$ will denote a smooth, finite-dimensional closed manifold with a basepoint $x_0\in X$.  We set $\pi_1 X = \pi_1 (X, x_0)$, and similarly for other based spaces.  Consider a smooth (right) principal $G$--bundle $\pi\co P \to X$, where $G$ is a Lie group, and equip $P$ with a basepoint $p_0 \in P_{x_0} := \pi^{-1} (x_0)$.  One may define a connection $A$ on $P$ to be a $G$--equivariant splitting of the natural map
\begin{equation}\label{connection}\xymatrix{ T(P) \ar[r] & \pi^* T(X)},\end{equation}
where $T(-)$ denotes the tangent bundle.
Geometrically, this corresponds to a $G$--invariant choice of horizontal direction in the bundle $P$.  We denote the space of all connections by $\A (P)$; this is in fact an affine space modeled on the vector space of $\textrm{ad} (P)$--valued 1--forms.  In particular, $\fc(P)$ is contractible.
The trivial bundle $X\cross G$ then has the trivial horizontal connection, and a connection is \emph{flat} if it is locally isomorphic to the trivial connection.  We denote the space of flat connections by $\fc(P)\subset \A (P)$.  Given a path $\gamma: I\to X$ with $\gamma(0) = x_0$, existence and uniqueness of solutions to ODE's produces a lift $\wt{\gamma}$ of the entire path, starting at $\wt{\gamma} (0) = p_0$.  Loops may not lift to loops, and but $\wt{\gamma} (0)$ and $\wt{\gamma} (1)$ lie in the same fiber and hence there is a unique element $g\in G$ such that $\wt{\gamma} (1) \cdot g = \wt{\gamma} (0)$.  When $A$ is flat, local triviality implies that this function $\gamma \mapsto g$ is well-defined on based homotopy classes, and in fact defines a homomorphism $\pi_1 X \to G$, known as the \emph{holonomy} representation $\mH (A)$.  The based gauge group $\G_0 (P)$ consists of all principal bundle automorphisms restricting to the identity on $P_{x_0}$, and $\G_0$ acts freely on the space of connections.  This action preserves the subspace of flat connections, and does not change the holonomy representation: $\mcH (\phi_* A) = \mcH (A)$ for all $\phi\in \G_0 (P)$, $A\in \fc(P)$.

To be precise, we will work with the Sobolev space $\flatc^{k, p} (P)$.  This space is the completion of the space of \emph{smooth} connections in the Sobolev norm $L^p_k$ (for some $p$ and $k$).  The precise choice of norm will not be important, although we do require enough regularity that the holonomy map
$$\mH \co \flatc(P) \to \Hom (\pi_1 X, G)$$
is well-defined and continuous.  
Similarly, the gauge group $\G_0 = \G_0^{k+1, p}$ will be the completion of the group of smooth, based principal-bundle automorphisms of $P$ in the $L^p_{k+1}$--norm.  (Increasing the regularity from $k$ to $k+1$ ensures that $\G_0$ acts continuously on $\fc$.)

\vspace{.15in}

We now describe a second construction, which (loosely speaking) provides an inverse to the holonomy map.  (For further details, see~\cite[Appendix]{Ramras-surface}).  Given a based smooth manifold $(Y, y_0)$ with universal cover $(\wt{Y}, \wt{y_0}) \srt{\pi} (Y, y_0)$ and a representation $\rho \co \pi_1 Y\to U(n)$, we can form the mixed bundle $E_\rho = \left(\wt{Y} \cross U(n)\right)/\pi_1 Y$, where $\pi_1 Y$ acts via 
$(\wt{y}, U) \cdot \gamma = (\wt{y} \cdot \gamma, \rho(\gamma)^{-1} U)$.  Note here that the action of $\pi_1 (Y)$ on $\wt{Y}$ is by deck transformations, and \emph{depends} on our choice of basepoint $\wt{y_0} \in \wt{Y}_{y_0}$.  The map $[\wt{y}, U] \stackrel{\pi_\rho}{\goesto} \pi(\wt{y})$  makes $E_\rho$ a principal $U(n)$--bundle, which is trivial over each simply connected subspace of $Y$.  Moreover, $E_\rho$ carries a canonical flat connection $A_\rho$, which descends from the trivial horizontal connection on $\wt{Y} \cross U(n)$ considered as a 
$\left(\pi_1 Y\cross U(n)\right)$--bundle.  Considered as a splitting of the map
$$T(E_\rho) \maps \pi_\rho^* (TY),$$
$A_\rho$ is simply the map
$$([\wt{y}, W], \vect{v}) \mapsto Dq \left(\left( (D_{\wt{y}} \pi)^{-1} (\vect{v}), \vect{0}_{W}\right)\right)$$
where $q \co \wt{Y} \cross U(n) \to E_\rho$ is the quotient map and $\vect{0}_{W}\in T_W U(n)$ is the zero vector.  (Note that the definition of $A_\rho$ does not depend on the specific representative $(\wt{y}, W)$ for $[\wt{y}, W]\in E_\rho$.)
The important property of this connection is that its holonomy representation, computed at the basepoint $[\wt{y_0}, I] \in E_\rho$, is precisely $\rho$.

We will need a naturality property for these connections.  
Given a based map $f \co (Y, y) \to (Z, z)$ between smooth manifolds, a representation $\rho \co \pi_1 Z \to U(n)$ induces a representation $f^*\rho = \rho \circ f_* \co \pi_1 Y \to U(n)$.  Moreover, if the universal covers $\wt{Y}$ and $\wt{Z}$ are equipped with basepoints $\wt{y}$ and $\wt{z}$ lying over $y$ and $z$, then there is unique based map $\wt{Y} \to \wt{Z}$ covering $f$, which induces a map of $U(n)$--bundles $\wt{f} \co E_{f^*\rho} \to E_\rho$.  Moreover, the connections $A_{f^*(\rho)}$ and $f^*(A_\rho)$ correspond under this isomorphism, in the sense that we have a commutative diagram
\begin{equation}\label{nat-conn}
\xymatrix{ T(E_{f^*\rho}) \ar[r]^{D\wt{f}} & T(E_\rho) \\
		\pi_{f^*\rho}^*TY \ar[r]^{(\wt{f}, Df)} \ar[u]^{A_{f^*\rho}} & \pi_\rho^*TZ \ar[u]^{A_\rho}.
	}
\end{equation} 
(In fact, there is always a unique connection $f^*(A_\rho)$ making such a diagram commute, and we write $A_{f^*\rho} = f^*(A_\rho)$.)

In order to study the connectivity of the map $B$ for surfaces, we will construct, for any $X$, a commutative diagram of the form
\begin{equation}\label{T-diag}
\xymatrix{\flatc \ar@{-->}[rr]^(.45)\T \ar[d]^\mH & & \Map_* (X, EU(n)) \ar[d]^{\pi_*} \\
	 	\Hom^0 (\pi_1 X, U(n)) \ar[r]^(.45){B} &  \Map_*^0 (B\pi_1 S, BU(n)) \ar[r]^(.52){f^*}_{\heq} 
					& \Map_*^0 (X, BU(n)).
	     }
\end{equation}
We now explain the various spaces and maps in this diagram.
First, $\flatc = \flatc^{k, p} (X\cross U(n))$ is the space of flat connections on the trivial $U(n)$--bundle over $X$ and $\Map_*^0$ denotes the connected component of the constant map in the based mapping space $\Map_*$.  The subspace $\Hom^0 \subset \Hom$ consists of those representations inducing trivial bundles, i.e. $\Hom^0$ is the inverse image, under $B$, of $\Map_*^0$.  Note that when $X=S$ is an aspherical surface, $\flatc$ is connected (Proposition~\ref{fc-conn}), so the image of $\mH$ always lies in the connected component 
$$\Hom_I (\pi_1 S, U(n))\subset \Hom^0(\pi_1 S, U(n))$$
containing the trivial representation (in fact, $\Hom_I = \Hom^0$ in this case).  Hence for surfaces we may replace $\Hom^0$ with $\Hom_I$ in Diagram (\ref{T-diag}).
The homotopy equivalence $f^*$ is induced by a map $f$ classifying the universal cover $p\co \wt{X} \to X$ as a principal $(\pi_1 X)$--bundle; in other words $f$ is chosen so as to fit into a commutative diagram
\begin{equation}\label{f}
\xymatrix{ \wt{X} \ar[d]^p \ar[r]^(.41){\wt{f}} & E\pi_1 X \ar[d] \\
		X \ar[r]^(.41){f} & B\pi_1 X.
		}
\end{equation}
We will choose $f$ such that $f(x_0) = *\in B\pi_1 X$.  We may now equip $\wt{X}$ with the basepoint $\wt{x_0} = \wt{f}^{-1} (*) \in p^{-1} (x_0)$.  
The mapping $\T$ will be defined in the proof of Theorem~\ref{B-highly-ctd}, and its definition depends on our choice of the map $f$.

\begin{lemma}$\label{fibs}$
The vertical maps in Diagram (\ref{T-diag}) are fibrations.
\end{lemma}
\begin{proof}  First consider the holonomy map $\mH$.
Recall that the based gauge group
$\G_0 = \G_0^{k+1, p} (X\cross U(n))$ acts continuously on $\fc$, and the quotient map is a principal $\G_0$--bundle (see Fine--Kirk--Klassen~\cite{FKK} or Mitter--Viallet~\cite{Mitter-Viallet}).  Moreover, the holonomy map $\mH$ factors through this quotient map and induces a homeomorphism
$$\Hom^0 (\pi_1 S, U(n)) \srm{\isom} \flatc/\G_0$$
(see Ramras~\cite[Section 3]{Ramras-surface}).
Hence the holonomy map itself is a principal $\G_0$--bundle, and in particular a fibration.

Next, say $(Z, z_0)$ is a based CW-complex and $(E, *) \srt{p}(B, *)$ is a (based) Serre fibration.  The induced map
$$p_*\co \Map_* (Z, E) \maps \Map_* (Z, B)$$
is a Serre fibration, because if $K$ is a CW complex, then any diagram
\begin{equation}
\xymatrix{ K \cross \{0\} \ar[r]^(.42)\alpha \ar[d] & \bMap (Z, E) \ar[d]^{p_*}\\
		   K\cross I \ar[r]^(.42)H \ar@{-->}[ur]^(.45){\wt{H}} & \bMap (Z, B)
		    }
\end{equation}  
can be filled in using the adjoint of a lifting in the induced diagram
\begin{equation}
\xymatrix{ K\cross I \cross \{z_0\}\cup K \cross \{0\} \cross Z \ar[r]^(.75){* \cup \alpha} \ar[d] & E \ar[d]^{p}\\
		   K\cross I \cross Z \ar[r] \ar@{-->}[ur] & B;
		    }
\end{equation}  
a lifting exists because the vertical map on the left is a homotopy equivalence.  These observations now apply to the map $\pi_*$ in Diagram (\ref{T-diag}), where $E = EU(n)$.  As noted at the start of this section, there is a canonical deformation retraction from $EU(n)$ to its basepoint, so the image of $\pi_*$ is precisely $\Map^0_*(Z, B)$.
\end{proof}

\begin{remark}   Atiyah and Bott~\cite[Section 2]{A-B} showed that the universal principal bundle for the group $\Map(S, U(n))$ is given by the map
$$\Map(S, EU(n)) \maps \Map^0 (S, BU(n)),$$
where the right-hand side denotes the space of nullhomotopic maps.  A similar argument could be used with based mapping spaces in place of unbased mapping spaces to obtain a model for the universal principal bundle with structure group $\Map_* (S, U(n))$.  For our purposes, however, the previous  observations will suffice.
\end{remark}

We will need a connectivity result for the space of flat connections over a surface.  

\begin{proposition}$\label{fc-conn}$
The space $\flatc(M^g \cross U(n))$ is precisely $2g(n-1)$--connected.  If $\Sigma$ is an aspherical, non-orientable surface with orientation double-cover $M^{\wt{g}}$, then $\flatc(\Sigma \cross U(n))$ is at least $(\wt{g}(n-1) -1)$--connected.
\end{proposition}

This result is an application of Morse theory for the Yang--Mills functional~\cite{A-B, Dask, Ho-Liu-non-orient, Rade}, together with Smale's infinite-dimensional transversality theorem (see Abraham~\cite{Abraham-Smale}).
For orientable surfaces of genus greater than one, the well-known connectivity estimate of $2(g-1)(n-2) -2$ would suffice; this estimate goes back at least to Daskalopoulos~\cite{Dask}.  The improved bound of $2g(n-1)$ was proven in Ramras~\cite[Proposition 4.9]{Ramras-surface}, along with the estimate for non-orientable surfaces.  In the orientable case this improved bound is in fact sharp, while in the non-orientable case the connectivity is precisely $2n\wt{g} -3 \wt{g} - 1$ (so long as $\wt{g}>1$ and $n\geqs 9$).  See Ramras~\cite[Theorems 4.9 and 4.11]{Ramras-YM}.  
\vspace{.15in}

We can now prove our connectivity result for the map $B$.  The first proof given below is the most direct approach to studying this map, and involves a new gauge-theoretical construction.  After the proof we will sketch a second approach, relying on ideas found in Donaldson-Kronheimer~\cite[Section 5.1]{DK}.  The two approaches are closely related, as we will see. 

Recall the notion of an $(l, k)$--connected map from Definition~\ref{k,l-ctd}. 

\begin{theorem}$\label{B-highly-ctd}$ Let $S$ be an aspherical surface.  Then the natural map
$$B\co \Hom(\pi_1 S, BU(n)) \maps \Map_* (B\pi_1 S, BU(n))$$
is at least $(\wt{g}(n-1))$--connected if $S$ is non-orientable (with double cover $M^{\wt{g}}$), and precisely $(1, 2g(n-1) + 1)$--connected if $S = M^g$ is orientable.
\end{theorem}

\begin{remark} The formulas from Ramras~\cite{Ramras-YM} (quoted above) give the precise connectivity of the map $B$ for most non-orientable surfaces.
\end{remark}

\nd {\bf Proof of Theorem~\ref{B-highly-ctd}.} The desired statements regarding the map
$$\pi_0 \Hom(\pi_1 S, BU(n)) \maps \pi_0 \Map_* (B\pi_1 S, BU(n))$$
are easily proven using the arguments in Ramras~\cite[Section 4]{Ramras-surface}, so we will restrict our attention to the identity component $\Hom_I(S, U(n))$ of the representation space and the component $\Map_*^0 (B\pi_1 S, BU(n))$  of nullhomotopic maps.  The restriction of $B$ to these subspaces will be denoted by $B_I$.  The remainder of the argument is the same in the orientable and the non-orientable cases.  

We will construct a continuous mapping $\T\co \flatc \maps \Map_* (S, EU(n))$
making the diagram (\ref{T-diag}) commute, so that by Lemma~\ref{fibs} we have a commutative diagram of fibrations
\begin{equation}\label{T-diag2}
\xymatrix{ \G_0^{k+1} \ar[r]^(.4)i_(.4){\heq} \ar[d] & \Map_* (S, U(n)) \ar[d]\\
		\flatc \ar@{-->}[r]^(.4)\T \ar[d]^\mH & \Map_* (S, EU(n)) \ar[d]^{\pi_*} \\
	 	\Hom_I (\pi_1 S, U(n)) \ar[r]^{f^*\circ B_I}  & \Map_*^0 (S, BU(n)).
	     }
\end{equation}

Since $\Map_* (S, EU(n))$ is contractible, the map $\T$ has connectivity one more than the space $\fc$.   We will identify the map between the fibers with the natural (continuous) inclusion $i\co\G_0^{k+1} \to \Map_* (S, U(n))$ guaranteed by the Sobolev embedding theorem, which is a weak equivalence by general approximation results for function spaces.  
Now the 5-lemma implies that $f^*\circ B_I$ and $\T$ have the same connectivity.   But $f^*$ is a homotopy equivalence, so $B_I$
 and $f^*\circ B_I$ have the same connectivity as well.  Thus the theorem will follow from Diagram (\ref{T-diag2}) and Proposition~\ref{fc-conn}.

We now construct the mapping $\T$.
The construction itself (unlike the connectivity calculation) does not require $S$ to be a surface; any compact manifold would suffice.  Recall that we denote the basepoint in $S$ by $x_0$, so $\pi_1 S = \pi_1 (S, x_0)$.  Given a connection $A\in \flatc(S\cross U(n))$, let $\rho = \mH(A)$ be the holonomy representation associated to $A$.  We can now construct the mixed bundle
$E_\rho$ described above, together with its flat connection $A_\rho$.  Since $\mH(A_\rho) = \rho$, there is a (unique) bundle isomorphism
\begin{equation}\label{phi_A} \phi_A\co S\cross U(n) \isom E_\rho \end{equation}
which sends $(x_0, I)$ to $[\wt{x_0}, I]$ and satisfies $\phi_A^* (A_\rho) = A$.  (We recall the construction of $\phi_A$ below; see Ramras~\cite[Appendix]{Ramras-surface} for full details.)  Note that the mixing construction can also be applied to the map $E\pi_1 S\to B\pi_1 S$, resulting in the bundle
$$E'_\rho = \left(E\pi_1 S\cross U(n)\right)/\pi_1 S \maps B\pi_1 S,$$
and we have a pullback diagram of principal $U(n)$--bundles
\begin{equation}\label{E_rho}
\xymatrix{ E_\rho \ar[r]^{\wt{f}_\rho} \ar[d] & E'_\rho \ar[d] \ar[r]^(.44){u_\rho} & EU(n) \ar[d] \\
		S \ar[r]^(.42) f & B\pi_1 S \ar[r]^(.46){B\rho} & BU(n),
	}
\end{equation}
where the map $\wt{f}_\rho\co E_\rho \to E'_\rho$ is given by $[\wt{x}, W] \mapsto [\wt{f} (\wt{x}), W]$ and the map
$u_\rho\co E'_\rho \to EU(n)$ is given by $[e, W]\mapsto E\rho (e) \cdot W$ (here $E\rho$ is the map in (\ref{E})).  Note that both of these maps are $U(n)$--equivariant.

We now have a commutative diagram
\begin{equation}\label{section}
\xymatrix{
S \cross U(n) \ar[r]^(.58){\phi_A} \ar[dr] & E_{\rho} \ar[d] \ar[r]^(.44){u_\rho \circ \wt{f}_\rho} & EU(n)\ar[d]\\
& S \ar@/^/[ul]^s 
\ar[r]^(.39){B\rho \circ f} & BU(n),
}
\end{equation}
where  $s\co  S \to S\cross U(n)$ is the canonical section $s(m) = (m, I)$.  We define $\T(A)$ as the composite $S\to EU(n)$ in this diagram:
$$\T(A) = u_\rho \circ \wt{f}_\rho \circ \phi_A \circ s.$$
Note that this is a based mapping from $S$ to $EU(n)$, since each of the component maps is based.

We must check that this defines a continuous map $\fc \to \Map_*(S, EU(n))$, and that Diagram (\ref{T-diag2}) commutes.  An examination of the definitions shows that the lower square in (\ref{T-diag2}) commutes.  Next, recall that by the Sobolev embedding theorem, each gauge transformation $\phi\in \G_0$ is a continuous map; we want to show that the map on fibers induced by $\T$ can be identified with this embedding of $\G_0$ into the group $\G_0^\textrm{cont}$ of continuous (based) automorphisms.  Let $A_0\in \fc$ denote the trivial horizontal connection on $S\cross U(n)$.  Then the map
$\G_0 \to \fc$, given by $\phi \mapsto \phi^{-1}_* A_0$, is a homeomorphism of $\G_0$ onto the fiber of $\mH$ over the trivial representation.  On the other hand, for each $\phi\in \G_0^\textrm{cont}$ we may write $\phi(x, I) = (x, \phi_2 (x))$ for some $\phi_2 \in \Map_* (S, U(n))$.  Now $\phi \goesto (x\goesto *\cdot \phi_2(x))$ defines a homeomorphism of 
$\G_0^\textrm{cont}$ onto the fiber of $\Map_*(S,EU(n)) \to \Map_* (S, BU(n))$ (over the constant map).
We just need to check that for any $\phi\in \G_0$, the map $\T(\phi^{-1}_* A_0)$ is simply $x\goesto *\cdot \phi_2 (x)$.  To understand the effect of $\T$ on the connection $\phi^{-1}_* A_0$, note that 
$\mH (A_0) = I$ (the trivial representation), and the bundle $E_I$ is canonically isomorphic to the trivial bundle.  Under this isomorphism, the connections $A_I$ and $A_0$ agree, and $\phi$ itself is the unique based isomorphism $S\cross U(n) \maps E_I$ taking $\phi^{-1}_* A_0$ to $A_I = A_0$.  
Since the mapping $EI\co E\pi_1 S \to EU(n)$ is the constant map to $*\in EU(n)$, 
the map $u_I \circ \wt{f}_I\co E_I = S\cross U(n) \to EU(n)$ in Diagram (\ref{section}) just sends $(x, W)$ to $*\cdot W \in EU(n)$.  One now sees that 
$\T(\phi^{-1}_* A_0)$ is precisely the map $x\goesto *\cdot \phi_2 (x)$.

To complete the proof, we must show that $\T$ is in fact continuous.
By a basic fact about the compact-open topology (Munkres~\cite[Theorem 46.11]{Munkres}), it suffices to check that the adjoint $\T^\vee\co \flatc \cross S \maps EU(n)$ is continuous.  Tracing the definitions, we see that the map $\T^\vee$ factors as
$$\fc \cross S \srm{F} \left(\Hom(\pi_1 S, U(n)) \cross \wt{S} \cross U(n)\right)/\pi_1 S \srm{\nu} EU(n),$$
where $F(A, x) = \phi_A (x, I)$ and $\nu ([\rho, \wt{x}, W]) = E\rho (\wt{f}\wt{x})\cdot W$. The map $\rho\mapsto E\rho$ is continuous (see (\ref{E})), and since $E\pi_1 S$ is locally compact and Hausdorff, the evaluation
map $\Map_*(E\pi_1 S, EU(n)) \cross E\pi_1 S \to EU(n)$ is continuous as well.  This establishes continuity of $\nu$.

To understand the mapping $F$, we need to recall the construction (see Ramras~\cite[Lemma 8.5]{Ramras-surface} for further details) of the map $\phi_A$ appearing in (\ref{phi_A}).
Given any principal $U(n)$--bundle $P\srt{\pi} S$ and a simply connected neighborhood $V$ of $x_0\in S$, we have a continuous parallel transport map
$$T \co \fc(P) \cross P|_V \maps P_{x_0}.$$
To define this map, choose a smooth path $\eta$ in $V$ from $v\in V$ to $x_0$.  Then for any $A\in \fc(P)$ and any $p\in P_x$, we define $T(A, p) = T_A (p)\in P_{x_0}$ by transporting $p$ back to the fiber over $x_0$ using the (unique) $A$--horizontal lift (starting at $p$) of $\eta$.  Flatness of $A$ and simple connectivity of $V$ imply that this map is well-defined.  
Continuity of $T \co \fc(P) \cross P|_V \to P_{x_0}$ follows from standard results on ordinary differential equations (Lang~\cite[Chapter IV]{Lang-dg}); note that as $v$ varies over a coordinate neighborhood inside $V$, we may vary the paths $\eta$ smoothly.
Similarly, parallel transport from the fiber over $x_0$ to $P|_V$ defines a continuous map 
$$\bar{T} \co \fc(P)\cross P_{x_0} \cross V \maps P|_V$$ 
defined by setting $\bar{T} (A, X, v) = \bar{T}_A (p, v)$ to be the endpoint of the $A$--horizontal lift (starting at $p$) of $\eta$.

The map $\phi_A$, restricted to $V\cross U(n)$, is given by
$$\phi_A (v, W) =  \bar{T}_{A_\rho}([x_0, \pi_2 T_A (v, W)], v),$$
where $\rho = \mH (A)$ and $\pi_2 \co S\cross U(n)\to U(n)$ is the projection.
Since $A$ and $A'$ have the same holonomy, $\phi_A$ is independent of the choices of paths, and is therefore well-defined on all of $S\cross U(n)$.

To prove continuity of $F$, we may restrict our attention to a simply connected neighborhood $V$ of $x_0$ (note that any $x\in S$ is contained in some such neighborhood). 
The inclusion $V\stackrel{i}{\injects} S$ induces a map of bundles
\begin{equation}\label{J}
\xymatrix{ \Hom(\pi_1 S, U(n)) \cross V\cross U(n) \ar[r]^(.44){J} \ar[d] 
	& \left( \Hom(\pi_1 S, U(n))\cross \wt{S} \cross U(n)\right)/\pi_1 S \ar[d] \\
		 \Hom(\pi_1 S, U(n)) \cross V \ar[r]^{\textrm{Id}\cross i} & \Hom(\pi_1 S, U(n))\cross S
		 }
\end{equation}
which restricts to the natural map $E_{i^*\rho} \to E_\rho$ at every representation $\rho \co \pi_1 S\to U(n)$.  Note that since $V$ is simply connected, $E_{i^*\rho}$ is the trivial bundle $V\cross U(n)$.  It now suffices to show that $J^{-1} \circ F_V$ is continuous, where $F_V$ is the restriction of $F$ to $\fc \cross V$.

We claim that $J^{-1} \circ F_V$ is in fact just the mapping
\begin{equation}\label{F_V} (A, v) \goesto (v, \pi_2 T_A (v, I)) \in V\cross U(n).\end{equation}
By the naturality property explained in Diagram (\ref{nat-conn}), the canonical connection $A_\rho$ on the bundle $E_\rho$ pulls back under $i$ to the \emph{trivial horizontal connection} on $V\cross U(n) = E_{i^*\rho}$.  Hence the bundle map (\ref{J}) simultaneously trivializes all of the connections $E_\rho$.
Formula (\ref{F_V}) now follows from the fact that pulling back connections commutes with parallel transport, and the fact that parallel transport along the trivial connection on $V\cross U(n)$ preserves the $U(n)$--coordinate.
$\hfill \Box$

\vspace{.25in}

We now explain another approach to Theorem~\ref{B-highly-ctd}, using results from Donaldson--Kronheimer~\cite[Section 5.1]{DK} regarding the \emph{universal bundle} for framed connections.  (I thank Ralph Cohen for drawing my attention to this reference.)  Although this approach gives an alternate proof of Theorem~\ref{B-highly-ctd}, the proof given above has further consequences and plays an important role in ongoing joint work with T. Baird.
We will assume the reader is familiar with the results from~\cite[Section 5.1]{DK}, which apply to arbitrary compact manifolds $X$.  Let $\A$ denote the space of \emph{all} connections on the trivial bundle over $X$.  Donaldson and Kronheimer construct a universal bundle $\mathcal{P}$ over $\A/ \G_0 \cross X$, and show that the adjoint of its classifying map\footnote{It may be somewhat optimistic to assume the existence of this classifying map, since it is unclear whether $\A/\G_0$ is paracompact.   Away from Yang-Mills moduli spaces Uhlenbeck's theorem~\cite[Theorem 3.6]{Uhl} gives only \emph{weak} convergence.  Moreover, standard metrization theorems cannot be applied so easily since it is unclear whether this space is regular (although continuity of the holonomy along loops proves that $\A/\G_0$ is at least Hausdorff).  Hence one should replace $\A/\G_0$ in this argument with a CW--approximation $K\stackrel{\heq}{\to} \A/\G_0$.  The pullback of $\mathcal{P}$ over $K$ will then admit the desired classifying map.
Our proof of Theorem~\ref{B-highly-ctd} avoids such technicalities.}
$u\co \A/\G_0\cross X \to BU(n)$ 
is a weak equivalence
$$u^\vee \co \A/\G_0 \maps \Map^0_* (X, BU(n)).$$
Rather than producing a map $\wt{\T}\co \mathcal{A} \to \Map_*(X, EU(n))$ analogous to our map $\T$, Donaldson and Kronheimer show that there are bijections 
$$\langle T, \A/\G_0\rangle \isom \langle T, \Map_*^0 (X, BU(n)\rangle$$
between the sets of basepoint-preserving homotopy classes of based maps.\footnote{Donaldson and Kronheimer in fact work only with \emph{unbased} mapping spaces, but since $\A/\G_0$ and $\Map^0_* (X, BU(n))$ are not simply connected in general, it is necessary to choose basepoints and work with pointed mapping spaces.  Their argument is easily modified to work in this case; the data of a ``family of framed connections" on a bundle $P\to T\cross X$ should be augmented to include a chosen trivialization of $P$ over $\{t_0\}\cross X$, which carries the connection over $\{t_0\}\cross X$ to the trivial connection on the trivial bundle, and isomorphisms between families should commute with this trivialization.} 

The space
$$\U_X = \left(\Hom(\pi_1 X, U(n)) \cross \wt{X} \cross U(n)\right)/\pi_1 X$$
is a principal $U(n)$--bundle 
over $\Hom(\pi_1 X, U(n))\cross X$, and
carries a framed family of \emph{flat} connections: the restriction of $\U_X$ to $\{\rho\} \cross X$ is just $E_\rho$, which carries the connection $A_\rho$.  Note that we have shown that this connection varies continuously over the representation space.
In this setting, universality means that this framed family of connections is classified by a mapping 
$$\Hom(\pi_1 X, U(n)) \cross X \maps \A/\G_0 \cross X.$$ 
A careful tracing of the definitions shows that this is just the map induced by the inclusion $\Hom(\pi_1 X, U(n)) \isom \fc/\G_0 \stackrel{i}{\injects} \A/\G_0$.  

When $X = S$ is an aspherical surface, this inclusion is  highly connected, since the quotient map for the based gauge group is a principal bundle, $\fc$ is highly connected (Proposition~\ref{fc-conn}), and $\A$ is contractible.  The bundle $\U_S$ is classified by the map $u\circ (i\cross \mathrm{Id})$, whose adjoint $u^\vee \circ i$ is highly connected.
Hence the adjoint of \emph{every} classifying map for $\U_S$ is highly connected.  Now, $\U_S$ is the pullback of the bundle
$$\U'_S  = \left( \Hom(\pi_1 S, U(n)) \cross E\pi_1 S \cross U(n)\right)/\pi_1 S \maps \Hom(\pi_1 S, U(n)) \cross B\pi_1 S$$
under the map $f\co S\to B\pi_1 S$ (as can be seen by considering the map $\U_S\to \U'_S$ defined via the maps $\wt{f}_\rho$ in (\ref{E_rho})).  
Moreover, $\U'_S$ is classified by the adjoint $B^\vee$ of the map $B$ (as can be seen by considering the maps $u_\rho$ in (\ref{E_rho})).

Hence the map
$$\Hom(\pi_1 S, U(n)) \cross S \xmaps{\id \cross f} \Hom(\pi_1 S, U(n)) \cross B\pi_1 S \srm{B^\vee} BU(n)$$
classifies $\U_S$, so its adjoint is highly connected.  But this adjoint fits into a commutative diagram
$$\xymatrix{ \Hom(\pi_1 S, U(n)) \ar[r]^(.43){B} \ar[dr] &  \bMap (B\pi_1 S, BU(n)) \ar[d]^(.45){f^*}_(.45){\heq}\\
		& \bMap (S, BU(n)),
		}
$$
so it has the same connectivity as $B$.

%%%%%%%%%%%%%%%%%%%%%%%%%%%%%%%%%%%%%%%%%%%%%%%%%%%
%%%%%%%%%%%%%%%%%%%%%%%%%%%%%%%%%%%%%%%%%%%%%%%%%%%

\section{Excision for connected sum decompositions of surfaces}$\label{excision}$

In this section we consider connected sum decompositions $S = S_1 \# S_2$, where the $S_i$ are compact surfaces other than the sphere.  We will prove an excision theorem that allows us to the build the deformation $K$--theory spectrum $\K(\pi_1 S)$ out of the deformation $K$--theory spectra of the \emph{free} groups appearing in the amalgamation diagram associated to our connected sum decomposition.
Work of Lawson~\cite{Lawson-simul} provides us with a very complete understanding of $\K$ for free groups, and in subsequent sections we will use Lawson's results to deduce the main theorems of this paper.

Our analysis of the mapping
$$B\co \Hom(\Gamma, U(n)) \maps \bMap(B\Gamma, BU(n))$$
in the previous section will be the key ingredient in our excision theorem (Theorem~\ref{excision-thm}). 
Unlike the constructions in the author's previous work~\cite{Ramras-surface}, the map $B$ is natural with respect to the group $\Gamma$, and hence behaves well with respect to amalgams.  The constructions in that paper, which relate the homotopy orbit space $\Hom\left(\pi_1 S, U(n)\right)_{hU(n)}$ to the unbased mapping space
$\Map(M, BU(n))$, pass through Sobolev spaces of connections, which are well-behaved only under smooth maps between manifolds of the same dimension.  Although these Sobolev spaces played an important role in our connectivity calculations for the map $B$, this map itself is independent of any geometric structure on the surface.

\subsection{Homotopy pullbacks and pushouts}

We will need some simple facts about homotopy pullbacks and pushouts.  Given a commutative square
\begin{equation}\label{square}
\xymatrix{ A \ar[r]^i \ar[d]^j & X \ar[d]^f\\
		    Y \ar[r]^g & Z
		    }
\end{equation}
of based spaces, let $\mcZ$ denote the homotopy pushout (i.e. the reduced double mapping cylinder) of the maps $Y \stackrel{j}{\longleftarrow} A \stackrel{i}{\maps} X$.
Then there is a natural map $\Psi\co \mcZ \to Z$ (which collapses both cylinders in $\mcZ$), and the square (\ref{square}) is said to be \emph{strongly homotopy co-cartesian} if $\Psi$ is a homotopy equivalence.

\begin{lemma}$\label{po-pb}$
For every strongly homotopy co-cartesian square of based spaces
\begin{equation}\label{hcc}\xymatrix{ A \ar[r]^i \ar[d]^j & X \ar[d]^f\\
		    Y \ar[r]^g & Z
		    }
\end{equation}
and every based space $W$, the induced diagram
\begin{equation}\label{hcm}
\xymatrix{ \bMap(Z, W) \ar[r]^{f^*} \ar[d]^{g^*} & \bMap(X, W) \ar[d]^{i^*}\\
		    \bMap(Y, W) \ar[r]^{j^*} & \bMap(A, W)
		    }
\end{equation}
is homotopy cartesian.
\end{lemma}
\begin{proof}  We first consider the case in which $i$ and $j$ are embeddings (which is the only case we will need).  Let $\mcZ$ denote the homotopy pushout of $i$ and $j$.  Since $i$ and $j$ are embeddings, the quotient map $X\coprod A\cross I \coprod Y \to \mcZ$ is proper, and hence $\Map_*(\mcZ, W)$ embeds as a subspace of the product
$$\bMap(X,W)\cross \bMap(A\cross I, W) \cross \bMap(Y, W).$$
This subspace is homeomorphic to the homotopy pullback $\mcP$ of the maps $i^*$ and $j^*$ (since we are working in the category of compactly generated spaces, where $\Map(A\cross I, W) \homeo \Map(I, \Map(A, W))$).
The map $\bMap(Z, W) \to \mcP$ can now be identified with the map induced by the projection $Z\srt{\heq} \mcZ$, which is a homotopy equivalence.

To deal with the general case, let $M_i$ and $M_j$ denote the reduced mapping cylinders of $i$ and $j$.  The square formed from (\ref{hcc}) by replacing $X$ and $Y$ with $M_i$ and $M_j$ (respectively) is still strongly homotopy co-cartesian.  There is a homotopy commutative cube connecting these two squares, and applying $\bMap(-, W)$ yields a homotopy commutative cube linking (\ref{hcm}) to the corresponding square with $X$ and $Y$ replaced by $M_i$ and $M_j$.  The maps between these squares are homotopy equivalences, and it can be checked that one of the squares is homotopy cartesian if and only if the other is (this requires analyzing the specific homotopies making the cube homotopy commutative, which in this case are quite simple).  Thus we have reduced to the case in which $i$ and $j$ are embeddings.
\end{proof}

We will need the following simple consequence of the Mayer--Vietoris sequence (\ref{M-V}) associated to a homotopy pullback. 

\begin{lemma}$\label{hpb-cube}$
Consider a commutative cube of spaces
$$
\xymatrix
{
A \ar[rr] \ar[rd]^\alpha \ar[dd] & {} & X \ar[rd]^\beta \ar'[d][dd] & {}\\
{} & A' \ar[rr] \ar[dd] & {} & X' \ar[dd]\\
Y \ar'[r][rr] \ar[dr]^\gamma & {} & Z \ar[rd]^\delta & {}\\
{} &Y' \ar[rr] & {} & Z'.
}
$$
in which the maps $\beta$, $\gamma$, and $\delta$ are weak equivalences and the front face (containing $A', X', Y'$ and $Z'$) is homotopy cartesian.  Then for any numbers $0\leqs l \leqs k \leqs \infty$, the map $\alpha$ is $(l,k)$--connected  (Definition~\ref{k,l-ctd}) if and only if the natural map
$$\eta\co A\maps \holim (Y\to Z \leftarrow X)$$
is $(l,k)$--connected.
\end{lemma}
\begin{proof}  We have a commutative diagram
$$\xymatrix{ A \ar[r]^\alpha \ar[d]^\eta & A' \ar[d]^{\eta'} \\
	\holim (Y\to Z \leftarrow X) \ar[r]^(.47){\wt{\alpha}}  & \holim (Y' \to Z' \leftarrow X').
	}
$$
The hypotheses imply that $\eta'$ and $\wt{\alpha}$ are weak equivalences, and it follows that $\alpha$ and $\eta$ have the same connectivity.
\end{proof}

\subsection{The excision theorem}

Consider a connected sum decomposition  $$S = S_1 \# S_2,$$ where the $S_i$ are surfaces other than the sphere.
Then we may write $S$ as a union of two surfaces with boundary $S'_1$ and $S'_2$, glued along their common boundary circle; of course $S'_i$ is just $S_i$ with an open disc removed.  The fundamental
groups of $S'_1$ and $S'_2$ are free groups, and the fundamental group of the boundary circle injects into these free groups: a map from $\bbZ$ to a free group is either trivial or injective, and if the map were trivial then $\pi_1 (S_i)$ would be isomorphic to the free group $\pi_1 S'_i$ (note that $S_i$ is the mapping cone of the inclusion $S^1\injects S_i$ of the boundary circle).  Letting $i\co \bbZ \injects F_k$ and $j\co \bbZ\injects F_l$ denote these inclusions, the Van Kampen theorem shows that $\pi_1 S$ is the pushout of the diagram
$$F_k \stackrel{i}{\hookleftarrow} \bbZ \stackrel{j}{\injects} F_l$$
where $k$ and $l$ are the ranks of $\pi_1 S'_1$ and $\pi_1 S'_2$ respectively.

\begin{theorem}$\label{excision-thm}$
Let $S_1$ and $S_2$ be compact surfaces other than the sphere, and let $S = S_1 \# S_2$.  Then with the notation established in the previous paragraph,
the map
$$\Phi\co \K( \pi_1 S) \maps \mathrm{holim} \left( \K (F_k) \stackrel{i^*}{\maps} \K (\Z)  \stackrel{j^*}{\longleftarrow}  \K (F_l) \right)$$
is injective on $\pi_0$ and an isomorphism on $\pi_*$ for $*>0$.  If $S$ is non-orientable, then $\Phi_*$ is an isomorphism on $\pi_0$ as well.
\end{theorem}
\begin{proof}  We will show that the mapping
\begin{equation}\label{Hueb}
\xymatrix{ \Hom (\pi_1 S, U(n)) = \Hom (F_k*_\bbZ F_l, U(n)) \ar[d]^{\Phi_n}\\
	 \holim \left( \Hom(F_k, U(n)) \stackrel{i^*}{\maps}
					      	\Hom(\bbZ, U(n))
						\stackrel{j^*}{\lmaps}
					      	\Hom(F_l, U(n))
		 \right)
		 }
\end{equation}
is $(1, 2g(n-1)+1)$--connected if $S = M^g$ and $(g(n-1))$--connected if $S$ is non-orientable with orientation double cover $M^{\wt{g}}$.  Since the $S_i$ are not spheres, we have $g\geqs 2$ and $\wt{g} \geqs 1$.  Hence the upper connectivity bounds tend to infinity with $n$, and applying Proposition~\ref{reduction-prop} completes the proof (note that all the monoids $\Rep(-)$ are stably group-like with respect to the trivial representations $1\in \Hom(-, U(1))$, by the results in Ramras~\cite[Section 4]{Ramras-surface}). 

To study the map $\Phi_n$, we will consider
the commutative diagram
\begin{equation}\label{B-diag}
\xymatrix@C=0pt{
\Hom(\pi_1 S, U(n)) \ar[rr] \ar[rd]^{B_{\pi_1 S}} \ar[dd] & {} & \Hom(F_l, U(n)) \ar[rd]^{B_{F_l}} \ar'[d][dd] & {}\\
{} & \bMap (B\pi_1 S, BU(n)) \ar[rr] \ar[dd] & {} & \bMap (BF_l, BU(n)) \ar[dd]\\
\Hom(F_k, U(n)) \ar'[r][rr] \ar[dr]^{B_{F_k}} & {} & \Hom(\bbZ, U(n)) \ar[rd]^{B_{\bbZ}} & {}\\
{} & \bMap (BF_k, BU(n)) \ar[rr] & {} & \bMap (B\bbZ, BU(n)).
}
\end{equation}

We claim that this diagram satisfies the hypotheses of Lemma~\ref{hpb-cube}.  This will show that the map $\Phi_n$ has the same connectivity as the map $B_{\pi_1 S}$, and the latter map has the desired connectivity by Theorem~\ref{B-highly-ctd}.

First, injectivity of the maps $i\co \bbZ \to F_k$ and $j\co \bbZ \to F_l$ implies that the diagram 
$$\xymatrix{ B\bbZ \ar[r]^j \ar[d]^i & BF_l \ar[d]\\
		    BF_k \ar[r] & B\pi_1 S
		    }
$$
is strongly homotopy co-cartesian (see, for example, Hatcher~\cite[Theorem 1B.11]{Hatcher}).
By Lemma~\ref{po-pb}, we conclude that the square of mapping spaces in diagram (\ref{B-diag}) is homotopy cartesian.  

We must show that the maps $B_{F_k}$, $B_{F_l}$, and $B_{\bbZ}$ in Diagram (\ref{B-diag}) are weak equivalences.  This is in fact true for any free group $F_m$, and could be proven using flat connections as in Theorem~\ref{B-highly-ctd}.  We prefer to give a more elementary explanation.  For any $m\geqs 1$, we have a homeomorphism $\Hom(F_m, U(n)) \homeo U(n)^m$, so the standard weak equivalence $U(n) \heq \Omega BU(n)$ (induced by the map $\Sigma U(n) \maps BU(n)$) yields a weak equivalence
$$U(n)^m \stackrel{\heq}{\maps}  \bMap \left(\bigvee_m S^1, BU(n) \right) = \prod_m  \bMap(S^1, BU(n)).$$
Similarly, each generator of $F_m$ corresponds to a 1-simplex, and hence a loop, in $BF_m$.  This yields a map $\bigvee_m S^1 \stackrel{\eta}{\injects} BF_m$, which is an isomorphism on $\pi_1$ and hence a homotopy equivalence.  The induced map
$$\Map_* (BF_m, BU(n)) \srm{j^*} \bMap \left(\bigvee_m S^1, BU(n) \right)$$
is then a homotopy equivalence as well.
We now have a commutative diagram
$$\xymatrix@C=50pt{ U(n)^m \ar[r]^(.39){B_{F_m}} \ar[dr]_\heq & \Map_* (BF_m, BU(n)) \ar[d]^\heq \\
					& \bMap\left(\bigvee_m S^1, BU(n)\right),
		}
$$
so $B_{F_m}$ is a weak equivalence as desired.
\end{proof}

\begin{remark}$\label{Kdef*}$
In Ramras~\cite{Ramras-surface}, the deformation $K$--theory of surface groups was computed by relating these groups to complex $K$--theory.  The groups $\K_*(\pi_1 S)$ can also be computed from the Mayer-Vietoris sequence associated to a connected sum decomposition $S = S_1 \# S_2$ (in the non-orientable case, the decomposition (\ref{amalg}) is simplest to use).  For this computation, one needs to understand the maps on $\K_*$ induced by the inclusions $\bbZ\injects \pi_1 S'_1$.  On $\K_0$, these maps induce the identity $\bbZ \srt{=} \bbZ$ because for free groups, $\K_0$ simply keeps track of the dimension of a representation.  On $\K_1$, Lawson has established a natural isomorphism $\K_1 (F_l) \isom \Hom(F_l, \bbZ)$.  Finally, Proposition~\ref{free-Bott} shows that the maps on the higher $\K_*$ are isomorphic to one of these two (depending on whether $*$ is odd or even).
\end{remark}

We note that the proof of Theorem~\ref{excision-thm} is related to work of Huebschmann~\cite{Huebs}, who proved that the homotopy pullback appearing in (\ref{Hueb}) is weakly equivalent to $\bMap(S, BU(n))$.  Huebschmann's results in fact apply to any two-dimensional complex, and for higher dimensional complexes he provides a subtler version of this homotopy pullback, which is again weakly equivalent to the based mapping space.  It would be interesting to know whether Huebschmann's map, when restricted to the representation space itself, actually agrees (up to homotopy) with the classifying map $B$.

%%%%%%%%%%%%%%%%%%%%%%%%%%%%%%%%%%%%%%%%%%%%%%%%%%%
%%%%%%%%%%%%%%%%%%%%%%%%%%%%%%%%%%%%%%%%%%%%%%%%%%%

\section{The stable moduli space of flat connections over a surface}$\label{coarse-moduli}$

We can now compute the homotopy type of the \emph{stable moduli space} of flat (unitary) connections over a surface.  By definition, this is the space
$$\Mf (S) =  \colim_{n\to \infty} \Hom\left(\pi_1 S, U(n)\right)/U(n)
\homeo \colim_{n\to \infty} \left( \flatc(S, n)/\isom \right)$$
where $\flatc(S, n)$ is the disjoint union, over all principal $U(n)$--bundles $P \to S$, of the space of flat connections on $P$.  The equivalence relation $\isom$ is induced by bundle isomorphisms.  It follows from Uhlenbeck compactness that the particular Sobolev topology on the spaces of flat connections on the bundles $P$ does not matter, since any reasonable choice yields a moduli space homeomorphic to $\Hom\left(\pi_1 S, U(n)\right)/U(n)$.  For details, see Ramras~\cite[Sections 3 and 8]{Ramras-surface}.

\subsection{The homotopy type of the stable moduli space}

\begin{theorem}$\label{moduli}$ Let $M^g$ denote an orientable surface of genus $g>0$.  Then there is a homotopy equivalence
$$\Mf (M^g) \heq \Sym^{\infty} (M^g).$$
If $\Sigma$ is a compact, aspherical, non-orientable surface, then
$$\Mf (\Sigma) \heq T^{k} \coprod T^k$$
where $k = \textrm{rank} (H^1 (\Sigma); \bbZ)$ and $T^k$ denotes the $k$--dimensional torus.
\end{theorem}

This theorem will follow from Lawson's work on the Bott map in deformation $K$--theory, together with our excision result for connected sum decompositions (Theorem~\ref{excision-thm}).

We begin by noting that for the case of the torus $M^1 = S^1\cross S^1$, Theorem~\ref{moduli} is elementary.  Here $\pi_1 (S^1\cross S^1) \isom \bbZ\cross \bbZ$, and there is actually a homeomorphism
$$\Hom\left(\bbZ \cross \bbZ, U(n)\right)/U(n) \stackrel{\homeo}{\maps} \Sym^n (S^1\cross S^1).$$
Each representation $\rho\co\bbZ\cross \bbZ\to U(n)$ corresponds to a pair $(A, B)$ of commuting unitary matrices.  Any such pair $(A,B)$ is simultaneously diagonalizable (by a unitary matrix), so there exists an orthonormal basis $e_1, \ldots, e_n \in \bbC^n$ such that each $e_i$ is an eigenvector of both $A$ and $B$.  Letting $\alpha_i$ and $\beta_i$ denote the eigenvalues of $A$ and $B$ (respectively) corresponding to the eigenvector $e_i$, the assignment
$$ (A, B) \mapsto [(\alpha_1, \beta_1), (\alpha_2, \beta_2), \ldots, (\alpha_n, \beta_n)] \in \Sym^n (S^1\cross S^1)$$
yields the desired homeomorphism (continuity is an exercise using Rouch\'{e}'s theorem, and since the left-hand side is compact and the right-hand side is Hausdorff, the map is in fact a homeomorphism).

\vspace{.25in}

In order to prove Theorem~\ref{moduli} in general, we need to discuss the Bott map in deformation $K$-theory.  First, recall that in~\cite{Lawson-prod}, Lawson constructed natural $E_\infty$--ring structures on the spectra $\K(-)$ arising from tensor products of representations.  Moreover, he produced a functor $\K_r$ from groups to ring objects in symmetric spectra together with a natural weak equivalence between 
$\K_r$ and the symmetric spectrum associated to $\K$.  (To be precise, Lawson's model for $\K$ is different from the one used in Ramras~\cite{Ramras-excision}.  For the comparison between these two models, see Ramras~\cite[Section 2]{Ramras-surface}.  Lawson's model arises from a $\Gamma$--space in the sense of Segal, and hence has an associated symmetric spectrum, as explained in Mandell--May--Schwede--Shipley~\cite{diagram-spectra}.)  We may then replace $\K_r (-)$ by a commutative $\bS$--algebra in the sense of Elmendorf--Kriz--Mandell--May~\cite{EKMM}, using Schwede's comparison functor~\cite{Schwede-comparison} from symmetric spectra to $\bS$--algebras.  Note that this functor is \emph{strongly} symmetric monoidal, and hence preserves the multiplicative structures on our spectra.  We will continue to denote the resulting $\bS$--algebra by $\K_r (-)$, or simply $\K(-)$.
(For a brief discussion of some of the basic features of $\bS$--modules and $\bS$--algebras, see Section~\ref{Atiyah-Segal}).  Since $\K(\{1\})$ is weakly equivalent to the ordinary connective $K$--theory spectrum $\ku$, the projection $\Gamma \to \{1\}$ makes the $\bS$--algebra $\K_r(\Gamma)$ into a $\ku$--algebra (or more precisely, an algebra over the rigidified $\bS$--algebra $\ku_r$.  We note that the passage between these various models for $\K$ preserves homotopy groups (i.e. maps $\bS^n \to \K(-)$ in the homotopy category, where $\bS^n$ is the $n^{\textrm{th}}$ suspension of the sphere spectrum $\bS^0$), since each of the functors between these various categories of spectra is a Quillen equivalence.

The Bott element $\beta \in \pi_2 \ku \isom \bbZ$ can be considered as a map
$$\beta\co \bS^2 \maps \ku,$$
where $\bS^2$ denotes the second suspension of the sphere spectrum $\bS = \bS^0$.
For any group $\Gamma$, this element corresponds to a class in $\pi_2 \K(\Gamma)$ via the natural map
$$\ku \heq \K(\{1\}) \maps \K(\Gamma).$$
Multiplication by this element gives us the Bott map $\beta$ in deformation $K$--theory.  More precisely,
$\beta = \beta_\Gamma$ is defined as the composite
$$\Susp^2 \K(\Gamma) = \bS^2 \sm \K(\Gamma) \stackrel{\beta\wedge \textrm{Id}}{\maps} \ku \wedge \K(\Gamma) \srm{\mu} \K(\Gamma),$$
where $\mu$ is the structure map of the $\ku$--module $\K(\Gamma)$.
It follows from the definitions that the Bott map is a natural transformation of functors from groups to $\ku$--modules.  

The following two results establish a close relationship between the Bott map and the stable moduli space $\Mf$.  

\begin{theorem}[Lawson \cite{Lawson-simul}, Corollary 4]$\label{Lawson-cofiber}$
For any finitely generated discrete group $\Gamma$, the homotopy cofiber of the Bott map $\beta_\Gamma\co \Susp^2 \K(\Gamma) \to \K(\Gamma)$ is weakly equivalent to the spectrum $\Rdef (\Gamma)$ associated to the topological abelian monoid
$$\coprod_{n=0}^\infty \Hom\left(\Gamma, U(n)\right)/U(n).$$
Hence there is a long exact sequence in homotopy
\begin{equation}\label{Lawson-seq}
\cdots \stackrel{\partial}{\maps} \K_*(\Gamma)  \stackrel{\beta_*}{\maps} \K_{*+2} \maps \Rdef_{*+2} (\Gamma)  \stackrel{\partial}{\maps} \K_{*-1} (\Gamma) \maps \cdots.
\end{equation}
\end{theorem}

The author employed this theorem in~\cite[Section 6]{Ramras-surface} to calculate $\Rdef_* (\pi_1 S)$ in low degrees ($S$ an aspherical surface), using computations of deformation $K$--theory.
Combining the group completion theorem and facts about the connected components of the representation spaces yields the following result.

\begin{lemma}[Ramras, \cite{Ramras-surface} Section 6] $\label{EM}$ For any aspherical surface $S$, the zeroth space of the spectrum $\Rdef(\pi_1 S)$ is weakly equivalent to $\bbZ \cross \Mf(S)$.
In particular, $\Mf(S)$ is weakly equivalent to a product of Eilenberg--MacLane spaces.
\end{lemma}

To apply our excision theorem, we need a basic fact about squares of spectra.

\begin{lemma}$\label{square-lemma}$
A square of spectra is homotopy $(l, k)$--cartesian if and only if its suspension (the square with all four spectra and all four maps suspended once) is homotopy $(l+1, k+1)$--cartesian.
\end{lemma}
\begin{proof} Using the homeomorphism $\holim ( \Omega Z \maps \Omega W \lmaps \Omega Y) \isom \Omega \holim (Z \maps W \lmaps Y)$, one sees that a square is homotopy $(l,k)$--cartesian if and only if its associated square of loop spectra is homotopy $(l-1, k-1)$--cartesian.  The proof is completed by considering the natural weak equivalence $X \stackrel{\heq}{\maps} \Omega \Susp X$.
\end{proof}

The following result determines the Bott map for free groups.

\begin{proposition}[Lawson \cite{Lawson-simul}, Section 5]$\label{free-Bott}$ For any free group $F_k$, there is a weak equivalence of $\ku$--modules
$$\K(F_k) \heq \ku \vee \bigvee_{k} \Sigma \ku,$$
and in particular the Bott map induces isomorphisms
$$\beta \co \K_i (F_k) \isom \K_{i+2} F_k$$
for all $i\geqs 0$.
\end{proposition}

Lawson in fact gives two proofs of this result.  The first approach makes use of the weak equivalence
$LBU(n) \heq U(n)^{\textrm{Ad}}_{hU(n)} = \Hom\left(\bbZ, U(n)\right)_{hU(n)}$
between the free loop space of $BU(n)$ and the homotopy orbit space of $U(n)$ under conjugation.  Another (less direct) route is to apply Theorem~\ref{Lawson-cofiber}.  Lawson has provided a direct computation of the spectrum $\Rdef (F_k)$ appearing in that theorem~\cite[Section 6]{Lawson-simul},  
and from this one may deduce injectivity of $\beta$ and compute the homotopy groups $\K_*(F_k)$.  The $\ku$--module structure of $\K(F_k)$ then follows as in the proof of Theorem~\ref{A-S}.

\begin{proposition}$\label{Bott-inj}$
For any Riemann surface $M^g$ with $g>0$, the Bott map
$$\beta_*\co  \K_{*} (\pi_1 M^g) \maps \K_{*+2} (\pi_1 M^g)$$
is an isomorphism for $*>0$ and an injection for $* = 0$.
Similarly, if $\Sigma$ is an aspherical, non-orientable surface, then
$$\beta_*\co  \K_{*} (\pi_1 \Sigma) \maps \K_{*+2} (\pi_1 \Sigma)$$
is an isomorphism for all $*\geqs 0$.
\end{proposition}
\begin{proof}
This result is a consequence of our excision theorem, Theorem~\ref{excision-thm} (except for the case of the torus $M^1$, which is easier and will be handled separately).  In the orientable case, any surface of genus $g>1$ may be written as a connected sum
$$M^g = M^{g-1} \# M^1.$$
The Bott map yields a commutative cube of spectra
\begin{equation}\label{Bott-cube}
\xymatrix@C=0pt{
\Susp^2 \K(\pi_1 M^g) \ar[rr] \ar[rd]^{\beta_{\pi_1 M^g}} \ar[dd] & {} & \Susp^2 \K(F_{2(g-1)}) \ar[rd]^{\beta_{F_{2(g-1)}}} \ar'[d][dd] & {}\\
{} &\K(\pi_1 M^g) \ar[rr] \ar[dd] & {} & \K(F_{2(g-1)})  \ar[dd]\\
\Susp^2 \K(\pi_1 F_2)  \ar'[r][rr] \ar[dr]^{\beta_{F_{2}}} & {} &\Susp^2 \K(\bbZ)  \ar[rd]^{\beta_{\bbZ}} & {}\\
{} &\K(\pi_1 F_2) \ar[rr] & {} & \K(\bbZ).
}
\end{equation}
The front square is homotopy $(1, \infty)$--cartesian
by Theorem~\ref{excision-thm}, so by Lemma~\ref{square-lemma} the back square is homotopy $(3, \infty)$--cartesian.  Since for any spectrum $X$ we have $\pi_* \Susp^2 (X) = \pi_{*-2} (X)$, the cube (\ref{Bott-cube}) induces a commutative diagram of Mayer--Vietoris sequences of the form
$$\xymatrix{
	\cdots \ar[r]^(.3)\partial 
	& \K_i (\pi_1 M^g) \ar[d]^{\beta_{\pi_1 M^g}} \ar[r]
	& \K_i (F_{2(g-1)}) \oplus \K_i (F_2) \ar[d]^{\beta_{F_{2(g-1)}} \oplus \beta_{F_2}}  \ar[r]
	& \K_i (\bbZ) \ar[d]^{\beta_{\bbZ}} \ar[r]^(.55)\partial  &\cdots\\
	\cdots \ar[r]^(.3)\partial & \K_{i+2} (\pi_1 M^g)\ar[r]
	& \K_{i+2}  (F_{2(g-1)}) \oplus \K_{i+2}  (F_2)\ar[r]
	& \K_{i+2}  (\bbZ) \ar[r]^(.53)\partial  &\cdots
		}
$$
(Note that this diagram exists only for $i>0$.)

By Proposition~\ref{free-Bott}, the Bott maps associated to the free groups $F_i$ induce isomorphisms in all degrees $i\geqs 0$.  Applying the 5-lemma now yields the desired result in dimensions at least 1.  Injectivity of $\beta_*\co \K_0 (\pi_1 M^g) \to \K_2 (\pi_1 M^g)$ was proven in Ramras~\cite[Proposition 6.6]{Ramras-surface}.

The case $g=1$ can be handled using Lawson's cofiber sequence (Theorem~\ref{Lawson-cofiber}  above), together with the above computation of the stable moduli space $\Mf (M^1)$ and Lemma~\ref{EM}. 

The non-orientable case is the same, but easier.  Any non-orientable surface is a connected sum of copies of $\bbRP^2$, so one obtains an amalgamation decomposition of $\pi_1 \Sigma$ to which Theorem~\ref{excision-thm} applies.  Moreover, there is no longer a failure of excision on $\pi_0$.
\end{proof}

In order to determine the \emph{homotopy type} of the stable moduli space (and not just its weak homotopy type) we need a lemma.

\begin{lemma}$\label{CW}$ The stable moduli space $\Hom(\Gamma, U)/U$ has the homotopy type of a CW complex so long as $\Gamma$ is finitely generated.
\end{lemma}
\begin{proof}  We begin by recalling that there is a homeomorphism 
$$\Hom(\Gamma, U)/U \homeo \colim_{n\to \infty} \Hom(\Gamma, U(n))/U(n).$$
The maps $\Hom(\Gamma, U(n-1))/U(n-1) \to \Hom(\Gamma, U(n))/U(n)$ defining this colimit are injective, since unitary representations admit unique decompositions into irreducibles.  Since these spaces are compact Hausdorff, the inclusions are in fact homeomophisms onto their images.

We will show that for each $n$, $\Hom(\Gamma, U(n))/U(n)$ can be given a triangulation $K_n\srt{\homeo} \Hom(\Gamma, U(n))/U(n)$ (with $K_n$ a finite simplicial complex) such that some subcomplex $L_n< K_n$ maps homeomorphically onto $\Hom(\Gamma, U(n-1))/U(n-1)$.  This will \emph{not} provide a triangulation of the colimit, because although the simplicial complexes $K_n$ and $L_{n+1}$ are both homeomorphic to $\Hom(\Gamma, U(n))/U(n)$, they need not be isomorphic to one another as simplicial complexes.  However, this will provide $\Hom(\Gamma, U)/U$ with the structure of an Ind--CW complex in the sense of Gutzwiller and Mitchell~\cite[Section 3.2]{Gutzwiller-Mitchell}.  As proven in~\cite[Proposition 3.2]{Gutzwiller-Mitchell}, all such spaces have the homotopy type of CW--complexes.  Another proof of this fact is that if $f\co K\srt{\heq} X$ is a weak equivalence from a CW complex to an Ind--CW complex (e.g. $K$ could be the realization of the singular complex of $X$) then the standard proof of the Whitehead Theorem (e.g. Hatcher~\cite[Theorem 4.5]{Hatcher}) provides a homotopy inverse for $f$.  (For a simple example of an Ind--CW complex which is not a CW complex, see~\cite{Gutzwiller-Mitchell}.)

To obtain the necessary triangulations, we will use a result of Illman~\cite[Theorem B]{Illman} regarding equivariant triangulations of subanalytic sets.  Let $G$ be a Lie group.  Illman shows that if $X$ is a subanalytic proper $G$--space (e.g. $X$ is a compact algebraic variety and $G$ is compact) and $Y\subset X$ is a $G$--invariant, subanalytic closed subset of $X$ (e.g. $Y$ is a closed subvariety of $X$), then $X/G$ admits a triangulation in which $Y/G$ is a subcomplex.  Illman's results go far beyond this, but for our purposes this simplified statement will suffice.

We now choose generators $g_1, \ldots, g_k$ for $\Gamma$.  To apply Illman's result, first note that for each $n$, the representation space $\Hom(\Gamma, U(n))$ is a real algebraic variety:  
$\Hom(\Gamma, U(n))$ is cut out from the real-algebraic variety $U(n)^k$ by the relations in $G$ (note that if $(A_1, \ldots, A_k)\in U(n)^k$ and $g_{i_1}^{t_i}\cdots g_{i_m}^{t_m} =1$ is a relation in $G$, then the corresponding matrix equation $A_{i_1}^{t_i}\cdots A_{i_m}^{t_m} =I$ is a system of polynomials in the entries of the $A_i$).  Let $X_{n-1}\subset \Hom(\Gamma, U(n))$ denote the union of the $U(n)$--orbits of $\Hom(\Gamma, U(n-1))$, and note that $X_{n-1}/U(n)$ and  $\Hom(\Gamma, U(n-1))/U(n-1)$ are equal as subsets of $\Hom(\Gamma, U(n))/U(n)$.  We claim that $X_{n-1}$ is a subvariety of $\Hom(\Gamma, U(n))$.  Representations $\rho: \Gamma\to U(n)$ conjugate to $U(n-1)$--representations are exactly those for which there exists a non-zero vector $v\in \bbC^n$ such that $(\rho(g_i) - I)v = 0$ for $i =1, \ldots, k$.  The required vector $v$ is thus a solution to a system of $nk$ linear equations in $n$ variables, and this system has a solution precisely when its coefficient matrix has rank less than $n$.  This occurs if and only if all of the $n\cross n$ minor determinants of the coefficient matrix are zero.  Each of these determinants is a polynomial in the entries of the matrices $\rho(g_i)$, and these ${nk\choose n}$ polynomials cut out $X_{n-1}$ from $\Hom(\Gamma, U(n))$.
Illman's theorem now shows that $\Hom(G, U(n))/U(n)$ admits a triangulation in which $X_{n-1}/U(n) = \Hom(G, U(n-1))/U(n-1)$ is a subcomplex, completing the proof.
\end{proof}

We can now prove our main result.

\vspace{.15in}

\noindent {\bf Proof of Theorem~\ref{moduli}.}  First we consider the orientable case.  The infinite symmetric product $\Sym^\infty (M^g)$ is an abelian topological monoid, hence weakly equivalent to a product of Eilenberg--MacLane spaces. By Lemma~\ref{EM}, the same is true of $\Mf (M^g)$.  As explained in Hatcher~\cite[Chapter 4K]{Hatcher}, the infinite symmetric product of a simplicial complex is a CW complex, and by Lemma~\ref{CW} we know that $\Mf(M^g)$ has the homotopy type of a CW complex.  Hence it will suffice to compute the homotopy groups  $\pi_* (\Mf(M^g))$ and show that they agree with $\pi_* \Sym^\infty (M^g)$.  By the Dold--Thom Theorem, the latter groups are precisely the reduced integral homology groups of $M^g$. For $i=1, 2$, the groups $\pi_i \Rdef (M^g)$ were computed in Ramras~\cite[Section 6]{Ramras-surface} and found to agree with the integral homology groups of $M^g$; these also agree with $\pi_* (\Mf(M^g))$ by Lemma~\ref{EM}.  

To complete the proof, we must show that for $*>2$ the groups $\pi_* (\Mf(M^g))\isom \Rdef_* (\pi_1 M^g)$ are zero.  This follows immediately from the fact that the Bott map is an isomorphism above degree zero (Proposition~\ref{Bott-inj}) together with Lawson's cofiber sequence (\ref{Lawson-seq}).
Since both $\Mf(M^g))$ and $ \Sym^\infty (M^g)$ are connected, the proof is complete.

For non-orientable surfaces $\Sigma$, the argument is essentially the same.  The groups $\pi_0 \Mf (\Sigma)$  and $\pi_1 \Mf (\Sigma)$ were computed in Ramras~\cite[Section 6]{Ramras-surface}.  Again (\ref{Lawson-seq}) and Proposition~\ref{Bott-inj} show that the higher homotopy of $\Rdef (\Sigma)$ is trivial.  $\hfill \Box$

\vspace{.25in}
The cohomology of the stable moduli space can be computed from Theorem~\ref{moduli}, and in particular it is torsion-free in both the orientable and the non-orientable cases.  Note that in the orientable case, $\bbC P^\infty$ is a model for the Eilenberg--MacLane space $K(\bbZ, 2)$, so $\Mf (M^g) \heq \Sym^\infty M^g \heq (S^1)^{2g} \cross \bbC P^\infty$,
and the cohomology now follows from the K\"{u}nneth Theorem.

\subsection{Explicit descriptions for non-orientable surface groups}$\label{explicit}$

In the case of a non-orientable surface, we can give several explicit descriptions of homotopy equivalences $\Mf( \Sigma) \stackrel{\heq}{\maps} T^k \coprod T^k.$

First, note that the inclusion 
$$i \co T^k \coprod T^k \isom \Hom\left(\pi_1 \Sigma, U(1)\right)/U(1)\injects \Hom(\pi_1 \Sigma, U)/U\
\homeo \Mf (\Sigma)$$
of the group of multiplicative characters
is split by the determinant map.  Hence on $\pi_1$, the determinant induces a surjection between finitely generated free abelian groups of the same rank, and must therefore be an isomorphism.  Since neither space has homotopy in higher dimensions, $\det$ is a homotopy equivalence and so is $i$.  

Following Lawson~\cite{Lawson-simul}, we will now consider another description of this equivalence in terms of eigenvalues.
Recall that by the Dold-Thom theorem, there is a weak equivalence $S^1\srt{\heq} \Sym^\infty (S^1)$, and hence $T^k \heq (\Sym^\infty S^1)^k$.  As noted above, the infinite symmetric product of a simplicial complex is a CW complex, so this weak equivalence is a homotopy equivalence.

We will write our surface $\Sigma$ in the form $\Sigma^g_i = M^g\# C_i$ ($g\geqs 1$), where $C_0 = \bbRP^2$ and $C_1$ is the Klein bottle (the case $g = 0$, i.e. $\Sigma = C_1$, is slightly different).  This gives an explicit choice of injective amalgamation diagram
\begin{equation}\label{amalg}
\xymatrix{ \bbZ \ar[r]^(.42){\eta_i} \ar[d]^m & \pi_1(C_i') \ar[d]^{\mu_i} \\
		F_{2g} \ar[r]^(.45){r} & \pi_1 \Sigma^g_i.
		}
\end{equation}
The map $m$ is the multiple commutator map (sending $1\in \bbZ$ to $\prod_1^g [a_i, b_i] \in F_{2g}$, where the $a_i$ and $b_i$ are free generators for $F_{2g}$), and the space $C'_i$ is $C_i$ minus a disk.  Hence $\pi_1 (C_0') \isom \bbZ$ and $\pi_1 (C_1') = F_2$, and the maps $\eta_i$ send
$1\in \bbZ$ to $d^2\in \pi_1 (C_0') = \langle d \rangle$ and $cdc^{-1}d\in \pi_1 (C'_1) = \langle c,d \rangle$ (respectively).  Diagram (\ref{amalg}) now yields maps
$$F_{2g+i} = F_{2g} * F_i \srm{r * \mu_i'} \pi_1 \Sigma^g_i$$
where $\mu_1'$ is the restriction of $\mu_1$ to the generator $c$, and $\mu'_0$ is the trivial map out of the trivial group.  (When $\Sigma = C_1$, we instead use the map $\bbZ = F_1 \to \langle c, d \, |\, cdc^{-1}d \rangle = \pi_1 C_1$ sending $1$ to $c$.)  

Lawson defines the stable eigenvalue map
$$\colim_n \Hom\left(\bbZ, U(n)\right)/U(n) \maps \Sym^\infty (S^1)$$
to be the map sending $\rho\co\bbZ \to U(n)$ to (unordered) collection of eigenvalues of $\rho(1)$.  Applying this map to each generator of the free group $F_k$ yields the stable eigenvalue map
$$\colim_n \Hom\left(F_k, U(n)\right)/U(n) \maps (\Sym^\infty (S^1))^k.$$
Finally, pulling back along the map $F_{2g+i} \srm{r*\mu'_i} \pi_1 \Sigma^g_i$ yields the stable eigenvalue map
$$\Mf (\Sigma^g_i) = \colim_n \Hom\left(\pi_1 \Sigma^g_i,  U(n)\right)/U(n) \maps (\Sym^\infty (S^1))^{2g+i} \heq T^{2g+i}.$$
This map can be extended to a map
$$E \co \Mf(\Sigma^g_i) \maps \bbZ/2\bbZ \cross  (\Sym^\infty (S^1))^{2g+i},$$
whose first coordinate sends a representation $\rho$ to Ho and Liu's obstruction class $o(\rho)\in \bbZ/2\bbZ$~\cite{Ho-Liu-ctd-comp-I, Ho-Liu-ctd-comp-II}.  (We note that $o(\rho)$ may be thought of as the first Chern class of the bundle $E_\rho$ from Section~\ref{universal}, which lies in 
$H^2(\Sigma^g_i; \bbZ) \isom \bbZ/2)$).

\begin{corollary}$\label{ev}$
Let $\Sigma^g_i$ be an aspherical, non-orientable surface.  Then the stable eigenvalue map
$$E \co \Mf(\Sigma^g_i) \maps \bbZ/2\bbZ \cross  (\Sym^\infty (S^1))^{2g+i}$$
is a homotopy equivalence.
\end{corollary}
\begin{proof}  The map on $\pi_0$ induced by $E$ is an equivalence by the results of Ho and Liu.
We have observed already that the inclusion $i\co\Mf(\Sigma^g_i, U(1)) \injects \Mf(\Sigma^g_i)$ is an equivalence, and the composite $E\circ i$ just sends $\rho$ to $(o(\rho), j(\rho))$
where the map $j\co(S^1)^{2g+i} \to  \left( \Sym^{\infty} (S^1)\right)^{2g+i}$ 
is the $(2g+i)$--fold product of the standard inclusion $\eta\co S^1 \injects \Sym^{\infty} (S^1)$.  Multiplication in $S^1$ gives a splitting of $\eta$, so $\eta$ induces an isomorphism on $\pi_1$ and is therefore a homotopy equivalence; it now follows that $j$, $E\circ i$, and $E$ are homotopy equivalences. 
\end{proof}

%%%%%%%%%%%%%%%%%%%%%%%%%%%%%%%%%%%%%%%%%%%%%%%%%%%
%%%%%%%%%%%%%%%%%%%%%%%%%%%%%%%%%%%%%%%%%%%%%%%%%%%

\section{The $\ku$--module structure of deformation $K$--theory}$\label{ku}$

In this section, we study deformation $K$--theory of surface groups and their products as modules over the connective $K$--theory spectrum $\ku$.  This leads to a refinement of the Atiyah--Segal theorem for surfaces (Ramras~\cite{Ramras-surface}) and to calculations of the stable moduli space $\Mf$ for products of surfaces.

Classically, the Atiyah--Segal theorem relates the representation ring $R(\Gamma)$ of a compact Lie group $\Gamma$ to the complex $K$--theory of its classifying space $B\Gamma$.  When $\Gamma$ is an infinite discrete group, non-isomorphic representations may sometimes be connected by a path, and such deformations give rise to isomorphic vector bundles.  To take the topology of the representation spaces into account, one replaces $R(\Gamma)$ by its homotopical analogue $\K(\Gamma)$, and attempts to relate this spectrum to the connective cover of the function spectrum $F(B\Gamma_+, \ku)$.  In~\cite{Ramras-surface}, the author established an isomorphism
\begin{equation}\label{A-S-isom}
\K_*(\pi_1 S) \isom K^{-*} (S)
\end{equation}
for any aspherical surface $S$.  In the non-orientable case, this result holds in all dimensions $*\geqs 0$; in the orientable case there is a failure in dimension zero.
In this section we provide a spectrum-level version of this result, relating the $\ku$--module $\K(\pi_1 S)$ to the function spectrum $F(S_+, \ku)$.  Using Lawson's product formula~\cite{Lawson-prod}, we extend this result to products of aspherical surfaces and compute the stable moduli space for such products.

The computations in this section show a striking similarity to the Quillen--Lichtenbaum conjectures, which predict isomorphisms between algebraic and \'{e}tale $K$--theory of schemes above the (\'{e}tale) cohomological dimension minus 2.  We find that for products of surfaces, deformation $K$--theory agrees with topological $K$--theory in dimensions greater than the rational cohomological dimension of the surface minus 2. 

Once one establishes a relationship between deformation $K$--theory and topological $K$--theory for a group $\Gamma$, the existence of the Atiyah--Hirzebruch spectral sequence
\begin{equation}\label{AHSS}
H^p(B\Gamma; K^q(*)) \conv K^{p+q} (B\Gamma)
\end{equation}
and Lawson's spectral sequence~\cite[Corollary 4]{Lawson-simul}
\begin{equation}\label{LAHSS}
\pi_p \Rdef (\Gamma) \otimes \pi_q \ku \conv \K_{p+q} \Gamma
\end{equation}
suggest that the homotopy groups of $\Rdef \Gamma$ should be closely related to the integral cohomology of $\Gamma$.  We will see that for a product $X$ of orientable surfaces, there is a very precise relationship: $\Rdef_*(\pi_1 X) \isom \wt{H}^*(X; \bbZ)$ in all degrees (Theorem~\ref{orient-prod}).  If $X$ is a product of non-orientable surfaces, there is a \emph{rational} isomorphism between $\Rdef_*(\pi_1 X)$ and  $\wt{H}^*(X; \bbZ)$ (Theorem~\ref{non-or-prod}), but the torsion can appear in different degrees (see Example~\ref{torsion}).

\subsection{A $\ku$--module version of the Atiyah--Segal theorem for surfaces}$\label{Atiyah-Segal}$

We will work in the setting of $\bS$--modules and modules over (commutative) $\bS$-algebras, as developed by Elmendorf, Kriz, Mandell, and May~\cite{EKMM}.  Throughout, $\bS$ will denote the sphere spectrum.  We now give a brief description of some of the main features of this theory, tailored to the subsequent applications and intended for those not familiar with the subject.  Although the technical underpinnings of~\cite{EKMM} are quite intricate, the theory has been developed to a point where it can be used quite readily.  In particular, most of the manipulations to follow are essentially elementary, and depend only on a limited number of formal properties of $\bS$--modules.

In classical stable homotopy theory, smash products of spectra are not well-behaved at the point-set level: in particular, one obtains associativity and commutativity of smash products up to homotopy only.  The category $\M_\bS$ of $\bS$--modules forms a replacement for the ordinary category of spectra, and admits a symmetric monoidal structure $\sm \co \M_\bS \cross \M_\bS \to \M_\bS$ called the smash product, which has the sphere spectrum $\bS$ as its unit object.  Although we will not attempt to describe the definition of $\bS$--modules precisely, we note that at the point-set level they are ``coordinate-free'' spectra with additional structure related to the linear isometries operad.  The category of $\bS$--modules admits a Quillen model structure, and the associated homotopy category (formed by inverting the weak equivalences, i.e. those maps inducing isomorphisms on homotopy groups) is equivalent to the classical stable homotopy category of spectra.  This equivalence respects the smash product~\cite[Theorem II.1.9]{EKMM}.  

In this rigidified setting, it makes sense to discuss algebras and their modules.  A (commutative) $\bS$--algebra is an $\bS$--module $R$ equipped with an associative, commutative, unital multiplication $\mu\co R\sm R \to R$.  Here the unital property refers to a specified unit map $\eta \co \bS \to R$.  The various axioms describing these algebraic assumptions can all be formulated diagrammatically, using the symmetric monoidal structure of the smash product.  
A (left) module over the algebra $R$ is defined as an $\bS$--module $M$ together with a multiplication $\mu \co R\sm M \to M$ satisfying the usual axioms.  The derived category of $R$--modules is formed by (formally) inverting the weak equivalences, i.e. those maps of $R$--modules inducing isomorphisms on all homotopy groups.

Loosely speaking, $\bS$--modules are analogous to abelian groups, and the smash product shares many formal properties with the ordinary tensor product 
$\otimes_\bbZ$ of abelian groups (or more precisely, with its derived functors $\Tor_*$).  These analogies can be seen upon passage to homotopy groups: for example, the homotopy groups of an $\bS$--algebra $R$ form a ring $\pi_* (R)$, and if $M$ is an $R$--module, then $\pi_* M$ is a module over $\pi_* R$.

Within the category of $R$--modules, there is another version of smash product, analogous to the tensor product $\otimes_A$ (and its derived functors $\Tor^A_*$) of modules over a ring $A$.
Given a right module $M$ and a left module $N$ over $A$, their tensor product can be described as the coequalizer of the two multiplication maps
$$\xymatrix{
P\otimes A \otimes Q  \ar@<1ex>[r]^{\Id \otimes \mu} \ar[r]_{\mu \otimes \Id} & P\otimes Q}.$$
When $M$ and $N$ are instead modules over the $\bS$--algebra $R$, replacing $\otimes$ with $\sm$ yields an analogous diagram, and $M\sm_R N$ is defined to be its coequalizer in $\M_\bS$.  The fact that coequalizers, and in fact all limits and colimits, exist in $\M_\bS$ is another important aspect of the theory.  Just as in ordinary algebra, bi-module structures on $M$ or on $N$ give rise to corresponding module structures on $M\sm_R N$.  Function $R$--modules (which we will not need) are defined similarly, and provide a right-adjoint to the smash product $\sm_R$.  Hence this smash product preserves colimits.   In the sequel, we will use this fact to distribute smash products over wedges (coproducts) of $R$--modules.

\vspace{.15in}
In order to describe the $\ku$--module $\K(\pi_1 \Sigma)$ when $\Sigma$ is non-orientable, we need some preliminaries.
Given an $\bS$--module $Y$ and a homotopy class $\alpha \in \pi_d Y$, there is a representing map of $\bS$--modules
$$\alpha\co \bS^d \maps Y,$$
where $\bS^d$ denotes (a cofibrant replacement for) the $d^{\textrm{th}}$ suspension of $\bS$ (see~\cite[II.1.8]{EKMM}; note that $\bS^0 \neq \bS$).  In particular, the element $n\in \pi_0 \bS = \bbZ$ gives rise to a map $d_n\co \bS^0 \to \bS$, and 
there is a corresponding map
\begin{equation}
\bS^0 \sm \ku \xmaps{d_n \sm \mathrm{Id}_\ku} \bS \sm \ku \isom \ku
\end{equation}
which we will denote by $\bd_n$.
The cofiber of this map is denoted $\ku/n$.  Since smashing preserves cofiber sequences~\cite[p. 59]{EKMM}, we have $\ku/n\isom \textrm{Cone} (d_n) \sm \ku$.
Hence $\ku/n$ inherits a $\ku$--module, and from the long exact sequence in homotopy associated to the cofiber sequence $\ku \to \ku \to \ku/n$, we see that $\pi_{2i} \ku/n \isom \Z/n\Z$ for $i\geqs 0$, and $\pi_* \ku/n = 0$ otherwise.  

\begin{theorem}$\label{A-S}$
Let $\Sigma$ be a compact, aspherical, non-orientable surface, and let $k$ denote the rank of $H^1 \Sigma$.  Then there are weak equivalences of $\ku$--modules
$$\K(\pi_1 \Sigma) \heq \ku \vee \bigvee_{k} \Susp \ku \vee \ku/2 \heq \wt{F}(\Sigma_+, \ku),$$
where $\wt{F}$ denotes the connective cover of the (based) function spectrum.

If $M^g$ is an orientable surface of genus $g>0$, then there are weak equivalences of $\ku$--modules
$$\K(\pi_1 M^g) \heq \ku \vee \bigvee_{2g} \Susp \ku \vee \Susp^2 \ku$$
and
$$\wt{F} (M^g_+, \ku) \heq \ku \vee \bigvee_{2g} \Susp \ku \vee \ku.$$
Hence the Bott map $\Susp^2 \ku \to \ku$ induces an isomorphism 
$$\K_* (\pi_1 M^g) \srt{\isom} \pi_* \wt{F} (M^g_+, \ku)$$
for $*>0$.
\end{theorem}

Recall that for $*\geqs 0$ and any space $X$, we have isomorphisms
\begin{equation}\label{ku-KU}\pi_* \wt{F} (X_+, \ku) \isom \ku^{-*} (X) \isom K^{-*}(X),\end{equation}
where $\ku^{-*} (X)$ is the cohomology theory represented by the spectrum $\ku$; note that these groups are also isomorphic to $\pi_* \bMap(X, \bbZ\cross BU)$.
(The first isomorphism in (\ref{ku-KU}) holds even for $*<0$, but since the spaces in the spectrum $\ku$ are Postnikov covers of the spaces in Bott spectrum $\mathbf{KU}$, the isomorphism 
$\ku^* (X) \isom \mathbf{KU}^* (X) =: K^*(X)$ holds only for $*\leqs 0$, or more generally when $X$ is a CW complex with no cells in dimensions $\leqs *$.)

The proof of Theorem~\ref{A-S} relies on the following simple lemma about $\ku$--modules (or more generally, modules over an $\bS$--algebra).  The unit $\eta\co \bS\to \ku$ induces a mapping
$$\bS^d \isom \bS \sm \bS^d \stackrel{ \eta \sm \Id}{\maps}  \ku \sm  \bS^d = \Sigma^d \ku.$$
Letting $1\in \bbZ = \pi_d \bS^d$ denote the canonical generator, we obtain a canonical element 
$$u = (\eta \sm \Id)_* (1) \in \pi_d (\Susp^d \ku).$$

\begin{lemma}$\label{rep}$ If $M$ is a $\ku$--module and $\alpha\in \pi_d M$, then there is a corresponding map of $\ku$--modules $f_\alpha \co \Susp^d \ku \to M$ such $f_\alpha (u) = \alpha$.  If $\alpha$ has order $n$, then this map extends to a map $\bar{f}_\alpha \co \Susp^d (\ku/n) \to M$ in the derived category of $\ku$--modules.
\end{lemma}
\begin{proof} The map $f_\alpha$ is simply the composite
$$\Susp^d \ku = \ku \sm \bS^d \xmaps{\textrm{Id} \sm \alpha} \ku \sm M \srm{\mu} M.$$
This is a map of $\ku$--modules because the following diagram commutes:
$$\xymatrix{
	\ku \sm \ku \sm \bS^d \ar[rr]^{\Id \sm \Id \sm \alpha} \ar[d]^{\mu \sm \Id} 
		& 
		&\ku \sm \ku \sm M \ar[r]^(.54){\Id \sm \mu} \ar[d]^{\mu \sm \id} 
		& \ku \sm M \ar[d]^\mu\\
	\ku \sm \bS^d \ar[rr]^{\id \sm \alpha} & & \ku \sm M \ar[r]^\mu & M.
	}
$$
(Note that commutativity of the right-hand square is the associativity axiom for the $\ku$--module $M$.)
To check that	$f_\alpha (u) = \alpha$, note that $f_\alpha (u)$ is classified by the map
$$\bS^d \isom \bS \sm \bS^d \xmaps{\eta \sm \id} \ku \sm \bS^d \xmaps{\id \sm \alpha} \ku \sm M \srm{\mu} M,$$
or equivalently 
$$\bS \sm \bS^d \xmaps{\id \sm \alpha} \bS \sm M \xmaps{\eta \sm \id} \ku \sm M \srm{\mu} M.$$
The unit axiom for the $\ku$--module $M$ says that $\mu \circ (\eta \sm \id) = \id_M$, so the composite is precisely $\alpha$.

Finally, say $\alpha$ has order $n$.  We must show that $f_\alpha$ extends over the mapping cone 
$$\Sigma^d(\ku/n) \isom \textrm{Cone}\left( \Sigma^d \ku \xmaps{\Sigma^d \bd_n} \Sigma^d \ku \right)$$
(note here that smash products with spaces, and in particular suspensions, commute with taking cofibers).  
We claim that the composition 
$$\Sigma^d \ku\xmaps{\Sigma^d \bd_n} \Sigma^d \ku \srm{f_\alpha} M$$
is nullhomotopic.  It follows from the definitions that this map is simply 
$$(\id_\bS \sm \mu) \circ (d_n \sm \id_{\ku} \sm \alpha)$$
and $d_n \sm \alpha$ is nullhomotopic since $\alpha$ has order $n$.  This nullhomotopy $H$ yields a map of cofiber sequences
$$\xymatrix{  \Sigma^d \ku\ar[r]^{\Sigma^d \bd_n} \ar[d] 
				& \Sigma^d \ku \ar[d]^{f_\alpha} \ar[r]
				& \Sigma^d(\ku/n) \ar@{.>}[d]^{\wt{f}_\alpha}\\
		   \textrm{Cone} (\Sigma^d \ku) \ar[r]^(.6)H 
		   		& M \ar[r]^(.4)j_(.4){\heq} 
				& \textrm{Cone} (H).
		}
$$
and $\bar{f}_\alpha = j^{-1} \wt{f}_\alpha$ is the desired map in the derived category.
\end{proof}

\noindent {\bf Proof of Theorem~\ref{A-S}.}
First we consider the orientable case, where (\ref{A-S-isom}) yields isomorphisms 
$\K_0 (\pi_1 M^g) \isom \bbZ$, $\K_1(\pi_1 M^g) \isom \bbZ^{2g}$, and $\K_2(\pi_1 M^g) \isom \bbZ^2$.  Let $\alpha_0\in \K_0(\pi_1 M^g)$ denote a generator, and let $\alpha_1^{1}, \ldots, \alpha_1^{2g}$ denote a basis for $\K_1(\pi_1 M^g)$.  The Bott map $\K_0 (\pi_1 M^g) \maps \K_2 (M^g)$ is injective with cokernel $\bbZ$ by~\cite[Proposition 6.6]{Ramras-surface}, so we may choose $\alpha_2 \in \K_2(\pi_1 M^g)$ such that $\K_2(\pi_1 M^g)$ is generated by $\beta(\alpha_0)$ and $\alpha_2$.  Lemma~\ref{rep} now yields a map of $\ku$--modules
\begin{equation}\label{A-S-map}
\ku \vee \bigvee_{2g} \Susp \ku \vee \Susp^2 \ku 
\xmaps{f_{\alpha_0} \vee \left( \bigvee_{i} f_{\alpha_1^i}\right) \vee f_{\alpha_2}} \K(\pi_1 M^g),
\end{equation}
which we claim is a weak equivalence.  This map is induces isomorphisms on $\pi_0$, $\pi_1$, and $\pi_2$ by construction.  On both sides of (\ref{A-S-map}), the Bott map induces isomorphisms $\beta_*\co \pi_i (-) \to \pi_{i+2} (-)$ for $i>0$, so we conclude by induction that our map is an isomorphism on homotopy in all degrees.  An analogous argument provides the desired map to the function spectrum.

The non-orientable case is similar, except that now we have an element of order two in $\K_0(\pi_1 \Sigma) \isom \pi_0 \wt{F}(\Sigma_+, \ku) \isom \bbZ\oplus \bbZ/2\bbZ$, so the representing map for this element extends over $\ku/2$.
$\hfill \Box$

\subsection{Products of surfaces}$\label{products}$

Combining the results of the previous section with Lawson's product formula
\begin{equation} \label{prod-form} \K(G\cross H) \heq \K(G) \sm_{\ku} \K(H) \end{equation}
(see~\cite{Lawson-prod}), we can compute deformation $K$--theory and the deformation representation ring $\Rdef$ for products of aspherical surfaces and circles.  
In addition to computing $\Rdef$ for products, we can also compute the stable moduli space: for this we need to know that $\Rep(\pi_1 X)$ is stably group-like, so that we can relate $\Mf$ and $\Rdef$.

\begin{remark} In (\ref{prod-form}), the derived smash product must be used, i.e. we first replacing one factor by a weakly equivalent cell module~\cite[IV.1]{EKMM}.  Such a cofibrant replacement always exists by~\cite[Theorem III.2.10]{EKMM}.  By~\cite[Theorem III.3.8]{EKMM}, this process produces a well-defined weak homotopy type (to compare different cofibrant replacements, one simply performs cofibrant replacement on both factors).  Similar comments apply to Lemma~\ref{ku/n} below.
\end{remark}

\begin{lemma}$\label{pi_0}$
Let $X$ be a product of aspherical surfaces and circles.  Then the monoid $\Rep(\pi_1 X)$ is stably group-like.
\end{lemma}
\begin{proof}   Let $X$ be a product of $d$ (aspherical) non-orientable surfaces and $k$ (aspherical) orientable surfaces or circles.  We will in fact show that for each representation $\rho\co \pi_1 X\to U(n)$, the representation $2^d \rho\co \pi_1 X\to U(2^d n)$ lies in the component of the trivial representation (so in the terminology of Section~\ref{zeroth-sp}, $\rho^* = (2^d -1)\rho$).  Note that when $d=0$, the conclusion is that all the representation spaces are connected.

We will need a general observation about direct products.
Any unitary representation of a discrete group is a sum of irreducible representations, and every irreducible representation $\rho\co G\cross H\to U(n)$ is isomorphic to a representation $\psi \otimes \phi$, with $\psi\co G\to U(n)$ and $\phi\co H\to U(n)$ irreducible (see, for example, Lawson~\cite[Lemma 37]{Lawson-prod}).
Irreducibility of $\psi$ and $\phi$ will not be important. 

The proof of the lemma now proceeds by induction on $d+k$; the case of a single surface was proven in~\cite[Section 4]{Ramras-surface} using results of Ho and Liu~\cite{Ho-Liu-ctd-comp-II} in the non-orientable case.  Assuming the result for $d+k$, we consider the case $X\cross S$ where $S$ is an aspherical surface or a circle and $X$ is a product of $d$ non-orientable surfaces and $k$ orientable surfaces or circles.  
Let $\epsilon = 0$ if $S$ is orientable, and let $\epsilon = 1$ if $S$ is non-orientable.  For
any representation $\rho\co \pi_1 X\cross \pi_1 S \to U(n)$, we have an isomorphism
$\rho \isom \oplus_i \psi_i \otimes \phi_i$, where the $\psi_i$ (respectively $\phi_i$) are unitary representations of $\pi_1 X$ (respectively $\pi_1 S$).
We now have:
\begin{eqnarray*}
2^{d+\epsilon} \rho  \isom 2^{d+\epsilon} \left(\bigoplus_i \psi_i \otimes \phi_i \right)  \isom
	\bigoplus_i 2^{d+\epsilon} (\psi_i \otimes \phi_i)\\
	 \isom \bigoplus_i (2^{d} \psi_i \otimes 2^\epsilon \phi_i).
\end{eqnarray*}
By induction, the representations $2^{d} \psi_i$ and $2^\epsilon \phi_i$ all lie in the component of the trivial representation.  Moreover, since $U(n)$ is connected, isomorphic representations lie in the same component, and this completes the proof.
\end{proof}

As in Lemma~\ref{EM}, we can now relate $\Rdef$ and $\Mf$ for products of surfaces.

\begin{lemma}$\label{EM2}$ Let $X$ be a product of surfaces and circles.  Then the zeroth space of the spectrum $\Rdef(\pi_1 X)$ is weakly equivalent to $\bbZ \cross \Mf(X)$.
In particular, $\Mf(X)$ is weakly-equivalent to a product of Eilenberg--MacLane spaces.
\end{lemma}

\begin{theorem}$\label{orient-prod}$ Let $X$ be a product of orientable surfaces and circles.  
Then the stable moduli space $\Mf(X)$ is homotopy equivalent to $\Sym^\infty (X)$, 
and there is a morphism 
$$\K(\pi_1 X) \maps F(X_+, \ku)$$
in the derived category of $\ku$--modules
that induces an isomorphism on $\pi_*$ for $*>\dim X - 2$.  Moreover, the groups $\K_{\dim X - 2} (\pi_1 X)$ and $\pi_{\dim X - 2} (\wt{F} (X_+, \ku)) \isom K^0 (X)$ are \emph{not} isomorphic.
\end{theorem}
\begin{proof} By Theorem~\ref{A-S}, we know that $\K(\pi_1 M^g)$ is (graded) free as a $\ku$--module, as is $\K(\pi_1 S^1) = \K(\bbZ) \heq \ku \vee \Susp \ku$ (Proposition~\ref{free-Bott}).  The statements about $\K$ now follow easily, by induction, from the fact that smash products distribute over wedges.  

Lawson's cofiber sequence shows that the homotopy of $\Rdef$ precisely measures the failure of the Bott map to be an isomorphism.  In this case, the Bott map is completely explicit, and one obtains an isomorphism
$$\pi_* \Rdef(\pi_1 X) \isom H_* (X)$$
by reading off the dimensions in which the wedge factors of $\K(\pi_1 X)$ appear.  The desired homotopy equivalence now follows from Lemma~\ref{EM2} and Lemma~\ref{CW}.
\end{proof}

When we bring in non-orientable surfaces, the situation is not as simple.  First, we need a lemma.

\begin{lemma}$\label{ku/n}$
For any $n, m\in \bbN$, there is a weak equivalence of $\ku$--modules
$$\ku/n \sm_{\ku} \ku/m \heq \ku/\gcd(n,m) \vee \Susp \ku/\gcd(n,m).$$
\end{lemma}
\begin{proof}  The proof is similar to that of Theorem~\ref{A-S}.  We begin by computing
$\pi_* (\ku/n \sm_{\ku} \ku/m)$.  In ordinary ring theory, one has an isomorphism
$$R/mR \otimes_R R/nR \isom (R/nR)/m(R/nR),$$
and smashing the cofiber sequence defining $\ku/m$ with the $\ku$--module $\ku/n$, we obtain, analogously, a cofiber sequence
\begin{equation}\label{gcd} (\bS^0 \sm \ku) \sm_\ku (\ku/n) \xrightarrow{\bd_m \sm \Id} \ku \sm_\ku (\ku/n) \maps \ku/m\sm_\ku \ku/n.
\end{equation}
(Note here that smashing over $\ku$ preserves cofiber sequences~\cite[p. 59]{EKMM}.)
Using the definition of $\ku\sm_\ku \ku/n$ as a coequalizer, we see that this module is simply $\ku/n$, and moreover the map $d_m\sm \Id$ is identified with 
$$\bS^0 \sm \ku/n \xrightarrow{d_m\sm \Id} \bS\sm \ku/n \isom \ku/n.$$
Hence on (positive) even-dimensional homotopy, $\bd_m\sm \Id$ is simply multiplication by $m$ on $\pi_{2k} (\ku/n)\isom Z/n\Z$, while in all other dimensions we have $\pi_{l} (\ku\sm_\ku \ku/n) \isom \pi_{l} (\ku/n) = 0$.  The long exact sequence in homotopy associated to (\ref{gcd}) now breaks down into exact sequences
$$0 \maps \Z/\gcd(m,n) \Z \maps \Z/n\Z \stackrel{\cdot m}{\maps} \Z/n\Z \maps \Z/\gcd(m,n)\Z \maps 0,$$ and we obtain isomorphisms $\pi_*\left(\ku/n \sm_\ku \ku/m\right) \isom \Z/\gcd(m,n)$ for $*\geqs 0$ (and $\pi_* \left(\ku/n \sm_\ku \ku/m \right) = 0$ for $*<0$).

By Lemma~\ref{rep}, choosing generators in $\pi_0 (\ku/n\sm_\ku \ku/m)$ and $\pi_1 (\ku/n\sm_\ku \ku/m)$ yields a map of $\ku$--modules
$$\ku/\gcd(m,n) \vee \Susp \ku/\gcd(m,n) \maps \ku/n\sm_\ku \ku/m,$$
and to show that this map is a weak equivalence it suffices to check that the Bott map is an isomorphism (in all non-negative degrees) on both of these $\ku$--modules.
For any $l\in \bbZ$, the Bott map $\pi_* \ku/l \to \pi_{*+2} \ku/l$ is an isomorphism for all $* \geqs 0$, as can be seen by applying the 5-lemma to the diagram of cofiber sequences
$$\xymatrix{ \Susp^2 \ku \ar[r]^{\Susp^2 m_l} \ar[d]^\beta & \Susp^2 \ku \ar[r] \ar[d]^\beta & \Susp^2 (\ku/l)\ar[d]^\beta \\
		    \ku \ar[r]^{m_l} & \ku \ar[r] & \ku/l.
		   }
$$
Similarly, the Bott map gives a map of cofiber sequences connecting the suspension of (\ref{gcd}) to (\ref{gcd}), completing the proof.
\end{proof}

\begin{remark} This result is similar in nature to~\cite[V.1.10]{EKMM}.  In comparing with~\cite[V.1]{EKMM}, it is important to keep in mind that the sequence $(m, n)$ in $\pi_0 \ku$ is \emph{not} a regular sequence in the ring $\pi_* \ku$, unless $m$ and $n$ are relatively prime.  In other words, when $\gcd(m,n)\neq 1$, the element $[m]\in \Z/n\Z$ is a zero-divisor.  
We note that Lemma~\ref{ku/n} can also be proven using the K\"{u}nneth spectral sequence~\cite[IV.6]{EKMM}.
\end{remark}

\begin{proposition}$\label{non-or-prod}$ Let $X$ be a product of aspherical orientable surfaces and circles, and let $Y$ be a product of $d$ non-orientable aspherical surfaces.  Then there is a morphism
$$\K(\pi_1 (X\cross Y)) \maps \wt{F} ((X\cross Y)_+, \ku)$$
in the derived category of $\ku$--modules that induces an isomorphism on $\pi_*$ for $*> \dim X + d - 2$.
Moreover, in dimension $\dim X + d -2$ the groups $\pi_* (\wt{F} (X_+, \ku))$ and $\K_* (X)$ are \emph{not} isomorphic.
\end{proposition}

\begin{remark} Note that $\dim X + d$ is precisely the \emph{rational} cohomological dimension of the classifying space $X\cross Y$.  Hence as mentioned above, the bound in Proposition~\ref{non-or-prod} bears a striking resemblance to the Quillen--Lichtenbaum conjectures.
\end{remark}

\begin{proof}
We will discuss the case when $X$ is trivial; the full result comes from combining this case with Theorem~\ref{orient-prod}.   We proceed by induction on $d$.  Let $\Sigma$ be a non-orientable surface, and $Y$ be a product of $d$ non-orientable surfaces for which the theorem holds.

The $\ku$--module $\K(Y\cross \Sigma)$ can be computed by induction using Lawson's product formula and Lemma~\ref{ku/n}, and the fact that smash products distribute over wedges.  If we let $r_0 (-)$ and $r_1(-)$ denote the total number of even- and odd-suspended $\ku$--summands in $\K(-)$,
then by Theorem~\ref{A-S} we have $r_0 (\Sigma) = 1$, and one finds that these numbers satisfy the recurrence
\begin{equation*}
r_0 (Y\cross \Sigma) = r_0 (Y) + r_1 (Y) r_1 (\Sigma)
\textrm{\,\,\,and\,\,\,}  r_1 (Y\cross \Sigma) = r_1 (Y) + r_0 (Y) r_1 (\Sigma).
\end{equation*}
If we let $t_0$ and $t_1$ denote the the total number of even- and odd-suspended 
$(\ku/2)$--summands in $\K(-)$) then $t_0 \Sigma = 1$, $t_1 \Sigma = 0$, and we find that
\begin{equation*}\begin{split}
t_0 (Y\cross \Sigma) = 2 t_0 (Y) + (r_1 (\Sigma) +1) t_1 (Y) + r_0 (Y) \\
\textrm{and\,\,\,} t_1 (Y\cross \Sigma) = 2 t_1 (Y) + (r_1 (\Sigma) +1) t_0 (Y) + r_1 (Y).
\end{split}
\end{equation*}

One can compute the groups $\pi_* \wt{F} ((X\cross Y)_+, \ku) \isom K^{-*} (X\cross Y)$ by a similar induction.  The K\"{u}nneth sequence in $K$--theory gives a (non-naturally) split exact sequence
$$0\maps K^*(Y) \otimes K^*(\Sigma) \srm{\alpha} K^*(Y\cross \Sigma) \srm{\gamma} \Tor (K^*Y, K^*\Sigma) \maps 0
$$
(see Bodigheimer~\cite{Bodigheimer}, for example).  Here we consider $K^* = K^0\oplus K^1$ as a $(\bbZ/2\bbZ)$--graded ring, and the maps $\alpha$ and $\gamma$ have degrees 0 and 1, respectively.  Also note that the $\Tor$ term is graded by giving $\Tor (K^pY, K^q \Sigma)$ grading $p+q$ (mod $2$).  
Letting $r'_0 (-)$ and $r'_1(-)$ denote the ranks of $K^0 (-)$ and $K^1 (-)$, 
and letting $t_0$ and $t_1$ denote the ranks over $\bbF_2$ of the torsion subgroups of $K^0$ and $K^1$ (which are always $\bbF_2$--vector spaces), we find that these numbers satisfy the same recurrence relations as $r_0, r_1, t_0$, and $t_1$.  The desired results now follows as in the proof of Theorem~\ref{A-S}.
\end{proof}

We can make some conclusions about the stable moduli space for products of non-orientable surfaces, although its relationship to cohomology is not as simple. 

\begin{proposition} Let $X$ be a product of aspherical surfaces and circles.  Then for any $*\geqs 0$, there is an isomorphism
$$\pi_* \Mf (X) \otimes \bbQ \isom \wt{H}^* (X; \bbQ),$$
and if $\Tor (-)$ denotes the torsion subgroup, then for $i = 0, 1$ we have isomorphisms
$$\bigoplus_{j \equiv i (\textrm{mod\,} 2)} \Tor (\pi_j \Mf (X)) \isom \bigoplus_{j \equiv i (\textrm{mod\,} 2)} \Tor (\wt{H}^j (X)).$$
Moreover, $\pi_* \Mf (X) = 0$ for $*$ greater than the rational cohomological dimension of $X$.
\end{proposition}
\begin{proof}  Again, we will just discuss the case in which $X$ is a product of non-orientable aspherical surfaces.  Recall (Lemma~\ref{EM}) that $\Rdef_*(\pi_1 X)$ and $\pi_* \Mf(X)$ agree for $*>0$; the difference in dimension zero is accounted for by using reduced cohomology.  
To show that the free summands in $\pi_* \Rdef$ and $\wt{H}^* (X; \bbZ)$
appear in the same dimensions, one proceeds by induction, noting that the $\bbZ$ factors in the homotopy of $\Rdef$ simply record the dimensions in which the wedge summands $\Susp^i \ku$ appear.  Similarly, the $\bbZ/2$ factors in the homotopy of $\Rdef$ record the dimensions in which the wedge summands $\Susp^i \ku/2$ appear.  One may deduce recurrence relations, similar to those in the proof of Proposition~\ref{non-or-prod}, describing 
$\bigoplus_{j \equiv i (\textrm{mod\,} 2)} \Tor (\pi_j \Mf (Y\cross \Sigma))$ in terms of the corresponding values for $Y$ and $\Sigma$ (where $Y$ is a product of aspherical surfaces and circles, and $\Sigma$ is an aspherical surface).  Using the K\"{u}nneth Theorem, one finds that the numbers  $\bigoplus_{j \equiv i (\textrm{mod\,} 2)} \Tor (\wt{H}^j (X))$ satisfy the same recurrences.
\end{proof}

\begin{example}$\label{torsion}$ The torsion in $\Rdef_* \Gamma$ (or equivalently $\pi_* (\Mf B\Gamma)$) can behave rather erratically in comparison with the cohomology of $\Gamma$.  For example, consider the product $(\bbRP^2 \# \bbRP^2)^3$.  The integral cohomology groups of this space, in increasing order of dimension, are:
$$\bbZ,  \,\,\bbZ^3,  \,\,\bbZ^3\oplus (\bbZ/2\bbZ)^3, \,\, \bbZ\oplus (\bbZ/2\bbZ)^9,  \,\,(\bbZ/2\bbZ)^{10},  \,\,(\bbZ/2\bbZ)^5, \,\, \bbZ/2\bbZ, \,\, 0, \,\, 0, \ldots . $$
On the other hand, the homotopy groups of $\Mf (\bbRP^2 \# \bbRP^2)^3$ are 
$$\bbZ \oplus  (\bbZ/2\bbZ)^7,  \,\,\bbZ^3 \oplus  (\bbZ/2\bbZ)^{14},  \,\, \bbZ^3 \oplus (\bbZ/2\bbZ)^7, \,\,\bbZ, \,\, 0, \,\, 0, \ldots.$$
\end{example}

\end{document}